\documentclass[11pt]{article}
\usepackage[margin=1.1in]{geometry}

\usepackage{graphicx}
\usepackage{epstopdf,epsfig}
\usepackage{amsmath}
\usepackage{mathtools} 
\usepackage{extarrows} 
 
\usepackage{comment}
\usepackage{palatino}
\usepackage{comment}
\usepackage[mathcal]{euler}
\usepackage{extarrows}
\usepackage{amsmath,amsthm,amsfonts,amssymb,latexsym,amscd,enumerate}
\usepackage{extarrows}
\usepackage{color}
\usepackage{scalerel,stackengine}
\stackMath
\newcommand\reallywidehat[1]{%
\savestack{\tmpbox}{\stretchto{%
  \scaleto{%
    \scalerel*[\widthof{\ensuremath{#1}}]{\kern-.6pt\bigwedge\kern-.6pt}%
    {\rule[-\textheight/2]{1ex}{\textheight}}
  }{\textheight}%
}{0.5ex}}%
\stackon[1pt]{#1}{\tmpbox}%
}
\theoremstyle{plain}
    \newtheorem{theorem}{Theorem}[section]
    \newtheorem{lemma}[theorem]{Lemma}
    \newtheorem{corollary}[theorem]{Corollary}
    \newtheorem{proposition}[theorem]{Proposition}

 \theoremstyle{definition}
    \newtheorem{definition}[theorem]{Definition}
    
    \newtheorem{example}[theorem]{Example}

    \newtheorem{remark}[theorem]{Remark}

\theoremstyle{remark}

\numberwithin{equation}{section}

\DeclareMathOperator{\Ad}{Ad}

\DeclareMathOperator{\ind}{index}

\DeclareMathOperator{\reg}{reg}

\DeclareMathOperator{\rank}{rank}

\DeclareMathOperator{\Spin}{Spin}

 \DeclareMathOperator{\Ind}{Ind}

\begin{document}

\newcommand{\Spinc}{\Spin^c}
\newcommand{\Todo}{\textbf{To do}}

    \newcommand{\R}{\mathbb{R}}
    \newcommand{\C}{\mathbb{C}} 
    \newcommand{\N}{\mathbb{N}}
    \newcommand{\Z}{\mathbb{Z}} 
    \newcommand{\Q}{\mathbb{Q}}
    \newcommand{\bK}{\mathbb{K}}

\newcommand{\g}{\mathfrak{g}}
\newcommand{\h}{\mathfrak{h}}
\newcommand{\p}{\mathfrak{p}}
\newcommand{\kg}{\mathfrak{g}} 
\newcommand{\kt}{\mathfrak{t}}
\newcommand{\kA}{\mathfrak{A}}
\newcommand{\XX}{\mathfrak{X}}
\newcommand{\kh}{\mathfrak{h}} 
\newcommand{\kp}{\mathfrak{p}}
\newcommand{\kk}{\mathfrak{k}}
\newcommand{\kq}{\mathfrak{q}}
\newcommand{\kl}{\mathfrak{l}}
\newcommand{\kn}{\mathfrak{n}}
\newcommand{\ks}{\mathfrak{s}}
\newcommand{\ka}{\mathfrak{a}}
\newcommand{\km}{\mathfrak{m}}
\newcommand{\tr}{\mathrm{Trace}}

\newcommand{\cc}{\circ}
\newcommand{\cE}{\mathcal{E}}
\newcommand{\cA}{\mathcal{A}}
\newcommand{\calL}{\mathcal{L}}
\newcommand{\calH}{\mathcal{H}}
\newcommand{\cO}{\mathcal{O}}
\newcommand{\cB}{\mathcal{B}}
\newcommand{\cK}{\mathcal{K}}
\newcommand{\cP}{\mathcal{P}}
\newcommand{\calD}{\mathcal{D}}
\newcommand{\cF}{\mathcal{F}}
\newcommand{\cX}{\mathcal{X}}
\newcommand{\calR}{\mathcal{R}}
\newcommand{\calM}{\mathcal{M}}
\newcommand{\calS}{\mathcal{S}}
\newcommand{\cU}{\mathcal{U}}

\newcommand{\Sj}{ \sum_{j = 1}^{\dim G}}
\newcommand{\Sk}{ \sum_{k = 1}^{\dim M}}
\newcommand{\ii}{\sqrt{-1}}

\newcommand{\ddt}{\left. \frac{d}{dt}\right|_{t=0}}

\newcommand{\PM}{P}
\newcommand{\DM}{D}
\newcommand{\LM}{L}
\newcommand{\vM}{v}

\newcommand{\Wedge}{\gamma}

\newcommand{\specialin}{\hspace{-1mm} \in \hspace{1mm} }

\newcommand{\beq}[1]{\begin{equation} \label{#1}}
\newcommand{\eeq}{\end{equation}}
\newcommand{\bspl}{\[ \begin{split}}
\newcommand{\espl}{\end{split} \]}

\newcommand{\Utilde}{\widetilde{U}}
\newcommand{\Xtilde}{\widetilde{X}}
\newcommand{\Dtilde}{\widetilde{D}}
\newcommand{\Etilde}{\widetilde{S}}
\newcommand{\wt}{\widetilde}
\newcommand{\Rep}{\text{Rep}}
\newcommand{\ess}{\mathrm{ess}}

\newcommand{\Rhat}{\widehat{R}}

\title{Higher Orbit Integrals, Cyclic Cocyles, and K-theory of Reduced Group $C^*$-algebra}

\author{Yanli Song\thanks{Department of Mathematics and Statistics, Washington University in St. Louis, St. Louis, MO, 63130, U.S.A., yanlisong@wustl.edu.},\ \ Xiang Tang\thanks{Department of Mathematics and Statistics, Washington University in St. Louis, St. Louis, MO, 63130, U.S.A., xtang@wustl.edu.}
}
\date{}
\maketitle

\abstract{Let $G$ be a connected real reductive group. Orbit integrals define traces on the group algebra of $G$. We introduce a construction of higher orbit integrals in the direction of higher cyclic cocycles on the Harish-Chandra Schwartz algebra of $G$. We analyze these higher orbit integrals via Fourier transform by expressing them as integrals on the tempered dual of $G$.  We obtain explicit formulas for the pairing between the higher orbit integrals and the $K$-theory of the reduced group $C^*$-algebra, and discuss their applications to representation theory and $K$-theory.}

\tableofcontents

\section{Introduction}
Let $G$ be a connected real reductive group, and $H$ a Cartan subgroup of $G$. Let $f$ be a compactly supported smooth function on $G$. For $x\in H^{\rm{reg}}$, the integrals
\[
\Lambda^H_f(x):=\int_{G/Z_G(x)} f(gxg^{-1})d_{G/H}\dot{g}\ \ \ \ 
\]
are important tools in representation theory with deep connections to number theory. Harish-Chandra showed the above integrals extend to all $f$ in the Harish-Chandra Schwartz algebra $\mathcal{S}(G)$, and obtained his famous Plancherel  formula \cite{MR0180629, MR0219666, MR0439994}. 

In this paper, we aim to study the noncommutative geometry of the above integral and its generalizations.  Let $W(H, G)$ be the subgroup of the Weyl group of $G$ consisting of elements fixing $H$. Following Harish-Chandra, we define the orbit integral associated to $H$ to be 
\[
F^H: \mathcal{S}(G)\to C^\infty(H^{\rm{reg}})^{-W(H,G)},\ F^H_f(x):=\epsilon^H(x) \Delta^G_H(x)\int_{G/Z_G(x)} f(gxg^{-1})d_{G/H}\dot{g},
\]
where $C^\infty(H^{\rm{reg}})^{-W(H,G)}$ is the space of anti-symmetric functions with respect to the Weyl group $W(H, G)$ action on $H$, $\epsilon^H(h)$ is a sign function on $H$, and $\Delta^G_H$ is the Weyl denominator for $H$.  Our starting point is the following property that for a given $h\in H^{\rm {reg}}$, the linear functional on $\mathcal{S}(G)$,  
\[
F^H(h): f\mapsto F_f^H(h), 
\] 
 is a trace on $\mathcal{S}(G)$, c.f. \cite{MR3990784}.  In cyclic cohomology, traces are special examples of cyclic cocycles on an algebra. In noncommutative geometry, there is a fundamental pairing between  periodic cyclic cohomology and $K$-theory of an algebra. The pairing between the orbit integrals $F^H(h)$ and $K_0(\mathcal{S}(G))$ behaves differently between the cases when $G$ is of equal rank  and non-equal rank. More explicitly, we will show in this article that when $G$ has equal rank, $F^H$ defines an isomorphism as abelian groups from the $K$-theory of $\mathcal{S}(G)$ to the representation ring of $K$, a maximal compact subgroup of $G$. Nevertheless, when $G$ has non-equal rank, $F^H$ vanishes on $K$-theory of $\mathcal{S}(G)$ completely, c.f. \cite{MR3990784}. Furthermore, many numerical invariants for $G$-equivariant Dirac operators in literature, e.g . \cite{Barbasch83, Connes82,  MR3990784, Mathai10, Wang14} etc, vanish when $G$ has non-equal rank. Our main goal in this article is to introduce generalizations of orbit integrals in the sense of higher cyclic cocycles on $\mathcal{S}(G)$ which will treat equal and non-equal rank groups in a uniform way  and give new interesting numerical invariants for $G$-equivariant Dirac operators. We remark that orbit integrals and cyclic homology of $\mathcal{S}(G)$ were well studied in literature, e.g. \cite{MR1186324, MR1840902, MR3286536, Pflaum15, Piazza-Posthuma, MR996447}. Our approach here differs from  prior works in its emphasis on  explicit cocycles.  
    
To understand the non-equal rank case better, we start with the example of the abelian group  $G{=}\mathbb{R}$, which turns out to be very instructive. Here  $\mathcal{S}(\mathbb{R})$ is the usual algebra of Schwartz functions on $\mathbb{R}$ with  convolution product, and it carries a nontrivial degree one cyclic cohomology.  Indeed we can define a cyclic cocycle $\varphi$ on $\mathcal{S}(\mathbb{R})$ as follows, (c.f. \cite[Prop. 1.4]{MR3286536}: 
\begin{equation}\label{eq:cocycle-R}
\varphi(f_0, f_1)=\int_{\mathbb{R}}xf_0(-x)f_1(x)dx .
\end{equation}
Under the Fourier transform, the convolution algebra $\mathcal{S}(\mathbb{R})$ is transformed into the Schwartz functions with pointwise multiplication, and the cocycle $\varphi$ is transformed into something more familiar: 
\begin{equation}\label{eq:fourier-R}
\hat{\varphi}(\hat{f}_0, \hat{f}_1)=\int_{\widehat{\mathbb{R}}} \hat{f}_0d\hat{f}_1. 
\end{equation}
It follows  from the Hochschild-Kostant-Rosenberg theorem   that $\hat{\varphi}$  generates the degree one cyclic cohomology of $\mathcal{S}(\widehat{\mathbb{R}})$, and accordingly $\varphi$ generates the degree one  cyclic cohomology of $\mathcal{S}(\mathbb{R})$. 

We notice that it is crucial to have the function $x$ in Equation (\ref{eq:cocycle-R}) to have the integral of $\hat{f}_0d\hat{f}_1$ on $\mathcal{S}(\widehat{\mathbb{R}})$.  And our key discovery is a natural generalization of the function $x$ on a general connected real reductive group $G$. Let $P=MAN$ be a parabolic subgroup of $G$. The Iwasawa decomposition $G=KMAN$ writes an element $g\in G$ as 
\[
g = \kappa (g)\mu(g) e^{H(g)}n \in KM A N = G. 
\]
Let $\dim(A)=n$. And the function $H=(H_1, \dots, H_n): G\to \mathfrak{a}$ provides us the right ingredient to generalize the cocycle $\varphi$ in Equation (\ref{eq:cocycle-R}). We introduce a generalization  $\Phi_{P}$ for orbit integrals in Definition \ref{defn:Phi}.  For $f_0, ..., f_n\in \mathcal{S}(G)$ and $x \in M$,  $\Phi_{P, x}$ is defined by the following integral, 
\[
\begin{aligned}
&\Phi_{P,x}(f_0, f_1, \dots, f_n) \\
:=&\int_{h \in M/Z_M(x)} \int_{K N} \int_{G^{\times n}} \sum_{\tau \in S_n}\mathrm{sgn}(\tau)\cdot H_{\tau(1)}(g_1...g_nk) H_{\tau(2)}(g_2...g_nk ) \dots  H_{\tau(n)}(g_n k) \\
& f_0 \big (kh x h^{-1}nk^{-1} (g_1\dots g_n)^{-1}\big)f_1(g_1) \dots f_n(g_n)dg_1\cdots dg_n dk dn dh,  
\end{aligned}
\]
where $Z_M(x)$ is the centralizer of $x$ in $M$. Though the function $H$ is not a group cocycle on $G$, we show in Lemma \ref{function H} that it satisfies a kind of twisted group cocycle property, which leads us to the following theorem in Section \ref{subsec:higher-cyclic}.\\

\noindent{\bf Theorem I.}
{\em 
(Theorem \ref{thm cocycle}) 
For a maximal parabolic subgroup $P_\cc= M_\cc A_\cc N_\cc$ and $x \in M_\cc$, the cochain $\Phi_{P_\cc, x}$ is a continuous cyclic cocycle of degree $n$ on $\mathcal{S}(G)$.  
}
\\

Modeling on the above example of $\mathbb{R}$, e.g. Equation (\ref{eq:fourier-R}), we analyzes the higher orbit integral $\Phi_P$ by computing its Fourier transform.   Using Harish-Chandra's theory of orbit integrals and character formulas for parabolic induced representations, we introduce in Definition  \ref{defn:phi-hat} a cyclic cocycle $\widehat{\Phi}_x$ defined as an integral on $\widehat{G}$. The following theorem in Section \ref{sec:fourier-transform} establishes a full generalization of Equation (\ref{eq:fourier-R}) to connected real reductive groups. \\

\noindent{\bf Theorem II.}
{\em 
(Theorem \ref{main thm-2} and \ref{main thm-3}) Let $T$ be a compact Cartan subgroup of $M_\cc$. 
For any $t \in T^\text{reg}$, and $f_0, \dots f_n \in \mathcal{S}(G)$, the following identity holds,
\[
\begin{split}
 \Phi_{P_\cc, e}(f_0, \dots, f_n)& =(-1)^n\widehat{\Phi}_{e}(\widehat{f}_0, \dots, \widehat{f}_n),\\
\Delta_{T}^{M_\cc}(t)\cdot  \Phi_{P_\cc, t}(f_0, \dots, f_n) &= (-1)^n\widehat{\Phi}_{t}(\widehat{f}_0, \dots, \widehat{f}_n),
\end{split}
\]
where $\widehat{f}_0,...,\widehat{f}_n$ are the Fourier transforms of $f_0,\cdots, f_n\in \mathcal{S}(G)$, and $\Delta_{T}^{M_\cc}(t)$ is the Weyl denominator of $T$.   
}
\\

As an application of our study, we compute the pairing between the $K$-theory of $\mathcal{S}(G)$ and $\Phi_{P}$. Vincent Lafforgue  showed in \cite{MR1914617} that Harish-Chandra's Schwartz algebra $\mathcal{S}(G)$ is a subalgebra of $C^*_r(G)$, stable under holomorphic functional calculus. Therefore, the $K$-theory of $\mathcal{S}(G)$ is isomorphic to the $K$-theory of the reduced group $C^*$-algebra $C^*_r(G)$. The structure of $C^*_r(G)$ is studied by \cite{MR3518312, MR894996}. As a corollary, we are able to explicitly identify \cite{Clare-Higson-Song-Tang} a set of generators of the $K$-theory of $C^*_r(G)$ as a free abelian group,  c.f. Theorem \ref{Schmid-iden}.  With wave packet, we construct a representative $[Q_\lambda]\in K(C^*_r(G))$ for each generator in Theorem \ref{Schmid-iden}.  Applying Harish-Chandra's theory of orbit integrals, we compute explicitly in the following theorem the index pairing between $[Q_\lambda]$ and $\Phi_P$. \\

\noindent{\bf Theorem III.}{ \em (Theorem \ref{thm:pairing}) 
The index pairing between cyclic cohomology and $K$-theory
\[
HP^{\rm{even}}\big(\mathcal{S}(G)\big)  \otimes K_0\big(\mathcal{S}(G)\big) \to \C
\] 
is given by the following formulas:
\begin{itemize}
	\item 
 We have 
\[
\langle \Phi_{e}, [Q_\lambda] \rangle =  \frac{(-1)^{\dim(A_P)}}{|W_{M_\cc \cap K}|} \cdot \sum_{w \in W_K} m\left(\sigma^{M_\cc}(w \cdot \lambda)\right),
\] 
where $\sigma^{M_\cc}(w \cdotp \lambda)$ is the discrete series representation with Harish-Chandra parameter $w \cdot \lambda$, and $m\left(\sigma^{M_\cc}(w \cdot \lambda)\right)$ is its Plancherel measure;
\item
For any $t \in T^\text{reg}$, we have that 
\[
\langle \Phi_{t}, [Q_\lambda] \rangle =  (-1)^{\dim(A_P)}\frac{\sum_{w \in W_K} (-1)^w e^{w \cdot \lambda}(t)}{\Delta^{M_\cc}_T(t)}.
\]
\end{itemize}
}

We refer the readers to Theorem \ref{thm:pairing} for the notations involved the above formulas. For the case of equal rank, the first formula was obtained in \cite{Connes82} in which Connes-Moscovicci used the $L^2$-index on homogeneous spaces to detect the Plancherel measure of discrete series representations. It is interesting to point out, c.f. Remark \ref{rmk:St}, that the higher orbit integrals $\Phi_{P_\cc, x}$ actually extend to a family of Banach subalgebras of $C^*_r(G)$ introduced by Lafforgue, \cite[Definition 4.1.1]{MR1914617}. However, we have chosen to work with the Harish-Chandra Schwartz algebra $\mathcal{S}(G)$ as our proofs rely crucially on Harish-Chandra's theory of orbit integrals and character formulas.

Note that the higher orbit integrals  $\Phi_{P, x}$ reduce to the classical ones when $G$ is equal rank. Nevertheless, our main results,  Theorem II. and III. for higher orbit integrals, are also new in the equal rank case.  For example, as a corollary to Theorem III., in Corollary \ref{cor:character}, we are able to detect the character information of limit of discrete series representations using the higher orbit integrals.  This allows us to identify the contribution of limit of discrete series representations in the $K$-theory of $C^*_r(G)$ without using geometry of the homogeneous space  $G/K$, e.g. the Connes-Kasparov index map.  As  an application, our computation of the index pairing in Theorem III. suggests a natural isomorphism $\cF^T$, Definition \ref{defn:isom} and Corollary \ref{cor:isom}, 
\[
\cF^T: K(C^*_r(G))\to \Rep(K). 
\]
In \cite{Clare-Higson-Song-Tang}, we will prove that $\cF^T$ is the inverse of the Connes-Kasparov index map,
\[
\Ind: \Rep(K)\to K(C^*_r(G)). 
\]

Our development of higher orbit integrals raises many intriguing questions. Let us list two of them here.
\begin{itemize}
\item Given a Dirac operator $D$ on $G/K$, the Connes-Kasparov index map  gives an element $\Ind(D)$ in $K(C^*_r(G))$. In this article, Theorem III. and its corollaries, we study the representation theory information of the index pairing
\[
\langle[\Phi_P], [\Ind(D)]\rangle.
\]
Is there a topological formula for the above pairing, generalizing the Connes-Moscovici $L^2$-index theorem \cite{Connes82}?
\item In this article, motivated by the applications in $K$-theory, we introduce $\Phi_{P,x}$ as a cyclic cocycle on $\mathcal{S}(G)$. Actually, the construction of $\Phi_{P,x}$ can be generalized to construct a larger class of Hochschild cocycles for $\mathcal{S}(G)$. For example, for a general (not necessarily maximal) parabolic subgroup $P$, there are corresponding  versions of Theorem I. and II., which are related to more general differentiable currents on $\widehat{G}$. How are these differentiable currents related to the Harish-Chandra's Plancherel theory? 
\end{itemize}

The article is organized as follows. In Section \ref{sec:prelim}, we review some basics about representation theory of real reductive Lie groups, Harish-Chandra's Schwartz algebra, and cyclic theory. We introduce the higher orbit integral $\Phi_P$ in Section \ref{sec:cyclic-cocycle} and prove Theorem I. The Fourier transform of the higher orbit integral is studied in Section \ref{sec:fourier-transform} with the proof of Theorem II. And in Section \ref{sec:higher-index}, we compute the pairing between the higher orbit integrals $\Phi_P$ and $K(C^*_r(G))$, proving Theorem III., and its corollaries. For the convenience of readers, we have included in the appendix some background  review about related topics in representation theory and $K(C^*_r(G))$. \\

\noindent{\bf Acknowledgments:} We would like to thank Nigel Higson, Peter Hochs, Markus Pflaum, Hessel Posthuma and Hang Wang for inspiring discussions. Our research are partially supported by National Science Foundation. We would like to thank Shanghai Center of Mathematical Sciences for hosting our visits, where parts of this work were completed. 
\section{Preliminary}\label{sec:prelim}
In this article, we shall not attempt to strive for the utmost generality in the class of groups
we shall consider. Instead we shall aim for (relative) simplicity.

\subsection{Reductive Lie group and Cartan subgroups}\label{subsec:reductive-Lie-group}
Let $G \subseteq GL(n, \R)$ be a self-adjoint group which is also the
group of real points of a connected algebraic group
defined over $\R$. For brevity, we shall simply say that $G$ is a \emph{real reductive
group}. In this case, the \emph{Cartan involution} on the Lie algebra $\kg$ is given by $\theta(X) = - X^T$, where  $X^{T}$ denotes the transpose matrix of $X$. 

Let $K = G \cap O(n)$,   which is a maximal compact subgroup of $G$. Let $\kk$ be the Lie algebra of $K$. We have a $\theta$-stable Cartan decomposition
\[
\kg = \kk \oplus \kp. 
\]

Let $A$ be a maximal abelian subgroup of positive definite matrices in $G$. And use $\ka$ to denote the associated Lie algebra of $A$. Fix a positive chamber $\ka_+ \subseteq \ka$. Denote the associated Levi subgroup by $L$. $L$ can be written canonically as a Cartesian product of two groups $M$ and $A$, i.e. $L=M\times A$. Let $N$ be the corresponding unipotent subgroup, and $P$ be the associated parabolic subgroup. We have   
\[
(M \times A) \ltimes N = M A N.
\] 

Let $T$ be a compact Cartan subgroup of $M$. Then $H = T A$ gives a Cartan subgroup of $G$. Note that there is a one to one correspondence between the Cartan subgroup $H$ and the corresponding parabolic subgroup $P$. We say that $H$ is the associated Cartan subgroup for the parabolic subgroup $P$ and vice versa. 

Let $\kh$ be the Lie algebra of $H$. Denote by $\calR(\kh, \kg)$ the set of roots. We can decompose $\calR(\kh, \kg)$ into a union of compact and non-compact roots:
\[
\calR(\kh, \kg) = \calR_c(\kh, \kg) \sqcup \calR_n(\kh, \kg),
\]
where $\calR_c(\kh, \kg) = \calR(\kh, \kk)$ and $\calR_n(\kh, \kg) = \calR(\kh, \kp)$. For a non-compact imaginary root $\beta$ in $\calR(\kh, \kg)$, we can apply the Cayley transform to $\kh$ and obtain a more non-compact Cartan subalgebra, together with the corresponding new parabolic subgroup. In this case, we say that the new parabolic subgroup is \emph{more noncompact than} the original one. This defines a partial order on the set of all parabolic subgroups of $G$.  

\begin{definition}\label{defn:max-parabolic}
If the associated Cartan subgroup of a parabolic subgroup is the most compact (i.e the compact part has the maximal dimension), then 
we call it the \emph{maximal parabolic subgroup}, and denote it by $P_\cc$. Accordingly, we write $P_\cc = M_\cc A_\cc N_\cc$ and the Cartan subgroup $H_\cc = T_\cc A_\cc$. 
\end{definition}

\subsection{Harish-Chandra's Schwartz function}\label{subsec:HCSchwartz}
Let $\pi_0$ be a spherical principal series representation of $G$, that is,  the unitary representation induced from the trivial representation of a minimal parabolic subgroup. Let $v$ be a unit $K$-fixed vector in the representation space of $\pi_0$, known as spherical vector. Let $\Xi$ be the matrix coefficient of $v$, i.e. for all $g \in G$,
\[
\Xi(g) = \langle v, gv \rangle_{\pi_0}.
\]
The inner product on $\kg$ defines a G-invariant Riemannian metric on $G/K$. For $g \in G$, let $\|g\|$ be the Riemannian distance from $eK$ to $gK$ in $G/K$. For every $m \geq 0, X, Y \in U(\mathfrak{g})$, and  $f \in C^{\infty}(G)$, set
\[
\nu_{X, Y, m}(f) :=\sup _{g \in G}\Big\{(1+\|g\|))^{m} \Xi(g)^{-1} \big|L(X) R(Y) f(g)\big|\Big\},
\]
where $L$ and R denote the left and right regular representations, respectively.

\begin{definition}
The Harish-Chandra Schwartz space $\mathcal{S}(G)$ is the space of $f \in C^\infty(G)$  such that for all $m \geq 0$ and $X, Y \in U(\mathfrak{g})$, $\nu_{X, Y, m}(f) < \infty$.
\end{definition}

The space $\mathcal{S}(G)$is a Fr\'echet  space in the semi- norms $\nu_{X, Y, m}$. It is closed under convolution, which is a continuous operation on this space. Moreover, if G has a discrete series, then all $K$-finite matrix coefficients of discrete series representations lie in $\mathcal{S}(G)$. It is proved in \cite{MR1914617} that $\mathcal{S}(G)$ is a $*$-subalgebra of the reduced group $C^*$-algebra $C^*_r(G)$ that is closed under holomorphic functional calculus.

\subsection{Cyclic cohomology}\label{subsec:cyclic-coh}
\begin{definition}
Let $A$ be an algebra over $\C$.  Define the space of Hochschild cochains of degree $k$ of $A$ by
\[
C^k(A) \colon = \mathrm{Hom}_\C\big(A^{\otimes(k+1)}, \C\big)
\]
of all bounded $k+1$-linear functionals on $A$. Define the Hochschild codifferential $b \colon C^k(A) \to C^{k+1}(A)$ by 
\[
\begin{aligned}
&b\Phi(a_0 \otimes \dots \otimes a_{k+1})  \\
=&\sum_{i=0}^k (-1)^i \Phi(a_0 \otimes \dots \otimes a_i a_{i+1} \otimes \dots \otimes a_{k+1}) + (-1)^{k+1} \Phi(a_{k+1}a_0 \otimes a_1 \otimes \dots \otimes a_{k}). 
\end{aligned}
\]
\end{definition}
The Hochschild cohomology is the cohomology of the complex $(C^*(A), b)$. 

\begin{definition}\label{defn:cyclic-coh}
We call a $k$-cochain $\Phi \in C^k(A)$ \emph{cyclic} if for all $a_0, \dots, a_k \in A$ it holds that 
\[
\Phi(a_k, a_0, \dots, a_{k-1}) = (-1)^k \Phi(a_0, a_1, \dots, a_k). 
\]
The subspace $C^k_\lambda$ of cyclic  cochains  is closed under the Hochschild codifferential. And the cyclic cohomology $HC^*(A)$ is  defined by the cohomology of the subcomplex of  cyclic cochains. 
\end{definition}

Let $R = (R_{i, j}), i, j = 1, \dots, n$ be an idempotent in $M_n(A)$. 
The following formula 
\[
\frac{1}{k!}\sum_{i_0,\cdots, i_{2k}=1}^n\Phi(R_{i_0i_1}, R_{i_1i_2}, ..., R_{i_{2k}i_0})
\]
defines a natural pairing between $[\Phi]\in HP^{\text{even}}(A)$ and $K_0(A)$, i.e. 
\[
\langle \ \cdot \ , \  \cdot  \ \rangle \colon   HP ^{\text{even}}(A) \otimes K_0(A)  \to \C. 
\]

\section{Higher orbit integrals}\label{sec:cyclic-cocycle}
In this section, we construct higher orbit integrals as cyclic cocycles on $\mathcal{S}(G)$ for a maximal parabolic subgroup $P_\cc$ of $G$.

\subsection{Higher cyclic cocycles}\label{subsec:higher-cyclic}
Let  $P = MAN$ be a parabolic subgroup and  denote $n = \dim A$. By the Iwasawa decomposition, we have that 
\[
G = KMAN. 
\] 
We define a map 
\[
H = (H_1, \dots, H_n) \colon G \to \ka
\] 
by the following decomposition
\[
g = \kappa (g)\mu(g) e^{H(g)}n \in KM A N = G.
\] 
The above decomposition is not unique since $K \cap M \neq \emptyset$. Nevertheless the map $H$ is well-defined. 
\begin{lemma}
\label{function H}
For any $g_0, g_1 \in G$, the function $H_i\left(g_1 \kappa(g_0)\right)$ does not depend on the choice of $\kappa(g_0)$. Moreover, the following identity holds
\[
H_i(g_0) + H_i\left(g_1 \kappa(g_0)\right) = H_i(g_1 g_0). 
\]
\begin{proof}
We write 
\[
g_0 = k_0m_0a_0 n_0, \hspace{5mm} g_1 = k_1m_1a_1 n_1. 
\]
Recall that  the group $MA$ normalizes $N$ and $M$ commutes with $A$. Thus, 
\[
H_i(g_1 k) = H_i(k_1m_1a_1 n_1 k) = H_i(k_1 k m_1a_1  n_1' ) = H_i(a_1) = H_i(g_1). 
\]
for any $k \in K \cap M$. It follows that $H_i
\left(g_1 \kappa(g_0)\right)$ is well-defined. Next, 
\[
H_i(g_1g_0) = H_i(a_1 n_1 k_0 a_0) =  H_i(a_1 n_1 k_0) + H_i(a_0). 
\]
The lemma follows from the following identities
\[
H_i(g_1 \kappa(g_0))  = H_i(a_1 n_1 k_0), \hspace{5mm} H_i(g_0) = H_i(a_0). 
\]
\end{proof}
\end{lemma}

 Let $S_n$ be the permutation group of $n$ elements. For any $\tau \in S_n$, let $\text{sgn}(\tau) = \pm1$ depending on the parity of $\tau$.
 \begin{definition}
 We define a function 
 \[
 C \in C^\infty\big(K \times G^{\times n}\big)
 \]
 by
 \[
 C(k, g_1, \dots, g_n): = \sum_{\tau \in S_n}\mathrm{sgn}(\tau)\cdot H_{\tau(1)}(g_1k) H_{\tau(2)}(g_2k ) \dots  H_{\tau(n)}(g_n k). 
 \] 
 \end{definition}

\begin{definition}\label{defn:Phi}
For any $f_0, \dots, f_n \in \mathcal{S}(G)$ and $x \in M$, we define a Hochschild cochain on $\mathcal{S}(G)$ by the following formula
\begin{equation}
\label{eq:higher cocycle}
\begin{aligned}
&\Phi_{P,x}(f_0, f_1, \dots, f_n)\colon =\int_{h \in M/Z_M(x)} \int_{K N} \int_{G^{\times n}} C(k, g_1g_2\dots g_n,  \dots, g_{n-1}g_n, g_n) \\
& f_0 \big (kh x h^{-1}nk^{-1} (g_1\dots g_n)^{-1}\big)f_1(g_1) \dots f_n(g_n)dg_1\cdots dg_n dk dn dh,  
\end{aligned}
\end{equation}
where $Z_M(x)$ is the centralizer of $x$ in $M$. 
\end{definition}

We prove in Theorem \ref{thm:definition-cocycle} that the above integral (\ref{eq:higher cocycle}) is convergent for $x \in M$. A similar estimate leads us to the following property.  

\begin{proposition}\label{prop:cont-cochain}
For all $x\in M$, the integral $\Phi_{P,x}$ defines a (continuous) Hochschild cochain on the Schwarz algebra $\mathcal{S}(G). $
\end{proposition}

In this paper, By abusing the notations, we write $\Phi_x$ for $\Phi_{P_\cc, x}$.  Furthermore, for simplicity, we omit the respective measures $dg_1, \cdots, dg_n$, $dk, dn, dh$, in the integral (\ref{eq:higher cocycle}) for $\Phi_x$. 

\begin{theorem}
\label{thm cocycle}
For a maximal parabolic subgroup $P_\cc= M_\cc A_\cc N_\cc$ and $x \in M_\cc$, the cochain $\Phi_{x}$ is a cyclic cocycle and defines an element 
\[
[\Phi_{x}] \in HC^n (\mathcal{S}(G)). 
\] 
\end{theorem}

\begin{remark}
In fact, our proof shows that $\Phi_{P, x}$ are cyclic cocycles for all parabolic subgroup $P$ and $x \in M$. For the purpose of this paper, we will focus on the case when $P = P_\cc$ is maximal, as $\Phi_{P_\cc, x}$, $x\in M$, generate the whole cyclic cohomology of $\mathcal{S}(G)$, c.f. Remark \ref{rmk:cyclic}. 
\end{remark}

\begin{remark}\label{rmk:St} We notice that our proofs in Sec. \ref{subsec:cocycle} and \ref{subsec:cyclic} also work for the algebra $\mathcal{S}_{t}(G)$ (for sufficiently large $t$) introduced in Definition \ref{defn:nut}. And we can conclude from Theorem \ref{thm:definition-cocycle} that $\Phi_x$ defines a continuous cyclic cocycle on $\mathcal{S}_{t}(G)\supset \mathcal{S}(G)$ for a sufficiently large $t$ for every $x
\in M$. 
\end{remark}
The proof of Theorem \ref{thm cocycle} occupies the left of this section. 

\subsection{Cocycle condition}\label{subsec:cocycle}
In this subsection, we prove that the cochain $\Phi_x$ introduced in Definition \ref{defn:Phi} is a Hochschild cocycle. 

We have the following expression for the codifferential of $\Phi_x$, i.e. 
\begin{equation}
\label{cycole equ-0}
\begin{aligned}
&b\Phi_{x}(f_0, f_1, \dots, f_n, f_{n+1})\\
=&\sum_{i=0}^n (-1)^i \Phi_{x} \big(f_0, \dots,  f_i \ast f_{i+1}, \dots,  f_{n+1}\big) + (-1)^{n+1} \Phi_{x}\big(f_{n+1}\ast f_0, f_1, \dots,   f_{n} \big).
\end{aligned}
\end{equation}
Here $f_i \ast f_{i+1}$ is the convolution product given by 
\[
f_i \ast f_{i+1}(h) = \int_G f_i(g) f_{i+1}(g^{-1}h) dg. 
\]
When $i = 0$, the first term  in the expression of $b\Phi_x$ (See Equation (\ref{cycole equ-0})) is computed by the following integral,
\begin{equation}
\label{cycole equ-1}
\begin{aligned}
&\Phi_{x} \big(f_0 \ast f_1, f_2 , \dots,  f_{n+1}\big)\\
=&\int_{M/Z_M(x)} \int_{K N} \int_G\int_{G^{\times n}} C(k, g_1g_2\dots g_n,  \dots, g_{n-1}g_n, g_n)  \\
&\cdot f_0(g) f_1 \big (g^{-1} \cdot kh x h^{-1}nk^{-1} (g_1\dots g_n)^{-1}\big) \cdot f_2(g_1) \dots f_{n+1}(g_n) .
\end{aligned}
\end{equation}
By changing variables, 
\[
t_1 = g^{-1} \cdot kh x h^{-1}nk^{-1} (g_1\dots g_n)^{-1}, \hspace{5mm}  t_j = g_{j-1},  \hspace{5mm} j = 2, \dots n+1, 
\]
we get 
\[
g = kh x h^{-1}nk^{-1} (t_1\dots t_{n+1})^{-1}. 
\]
We can rewrite (\ref{cycole equ-1}) into 
\begin{equation}
\label{cycole equ-2}
\begin{aligned}
&\Phi_{x} \big(f_0 \ast f_1, f_2 , \dots,  f_{n+1}\big)\\
=&\int_{M/Z_M(x)}  \int_{K N} \int_{G^{\times (n+1)}} C(k, t_2t_3\dots t_{n+1}, \dots, t_nt_{n+1},  t_{n+1})  \\
& f_0\big(kh x h^{-1}nk^{-1} (t_1\dots t_{n+1})^{-1}\big) f_1(t_1) \cdot f_2(t_2) \dots f_{n+1}(t_{n+1}) . 
\end{aligned}
\end{equation}
For $ 1\leq i \leq n$, we have
\begin{equation}
\label{cycole equ-3}
\begin{aligned}
&\Phi_{x} \big(f_0, \dots,  f_i \ast f_{i+1}, \dots,  f_{n+1}\big)\\
=&\int_{ M/Z_M(x)} \int_{K N} \int_G\int_{G^{\times n}} C(k, g_1g_2\dots g_n,   \dots, g_{n-1}g_n, g_n)  \\
&f_0 \big (khxh^{-1}nk^{-1} (g_1\dots g_n)^{-1}\big)f_1(g_1) \dots f_{i-1}(g_{i-1}) \\
&\big( f_i(g) f_{i+1} (g^{-1} g_i)  \big) \cdot f_{i+2}(g_{i+1}) \dots f_{n+1}(g_n).
\end{aligned}
\end{equation}
Let $t_j = g_j$ for $j = 1, \dots  i-1$, and
\[
t_i = g, \hspace{5mm} t_{i+1} = g^{-1}g_i, \hspace{5mm} t_j = g_{j-1}, \hspace{5mm} j = i+2, \dots, n+1. 
\]
We  rewrite (\ref{cycole equ-3}) as
\begin{equation}
\label{cycole equ-4}
\begin{aligned}
&\Phi_{x} \big(f_0, \dots,  f_i \ast f_{i+1}, \dots,  f_{n+1}\big)\\
=&\int_{ M/Z_M(x)} \int_{K N}\int_{G^{\times (n+1)}} C(k, t_1t_2\dots t_{n+1}, \dots, (t_{i+1} \dots t_{n+1})^{\string^}, \dots, t_{n+1})  \\
&\cdot f_0 \big (khxh^{-1}nk^{-1} (t_1\dots t_{n+1})^{-1}\big)f_1(t_1)  \dots f_{n+1}(t_{n+1}),
\end{aligned}
\end{equation}
where $(t_{i+1} \dots t_{n+1})^{\string^}$ means that the term is omitted in the expression.

Now we look at the last term in the expression of $b\Phi_x$, (c.f. Equation (\ref{cycole equ-0})), 
\begin{equation}
\label{cycole equ-5}
\begin{aligned}
&\Phi_{x} \big(f_{n+1} \ast f_0, f_1 , \dots,  f_{n}\big)\\
=&\int_{ M/Z_M(x)}\int_{K N} \int_G\int_{G^{\times n}} C(k, g_1g_2\dots g_n,   \dots, g_{n-1}g_n, g_n)  \\
&\cdot f_{n+1}(g) f_0 \big (g^{-1} \cdot k hmh^{-1}nk^{-1} (g_1\dots g_n)^{-1}\big) \cdot f_1(g_1) \dots f_{n}(g_n). 
\end{aligned}
\end{equation}
As before, we denote by 
\[
 t_j = g_{j}, \hspace{5mm} j = 1, \dots, n, 
\]
and $t_{n+1} = g$. We can rewrite Equation (\ref{cycole equ-5}) as 
\begin{equation}
\label{cycole equ-6}
\begin{aligned}
&\Phi_{x} \big(f_{n+1} \ast f_0, f_1 , \dots,  f_{n}\big)\\
=&\int_{ M/Z_M(x)}\int_{K N} \int_{G^{\times (n+1)}} C(k, t_1t_2\dots t_n, \dots, t_{n-1}t_n, t_n)  \\
&\cdot f_0 \big (t_{n+1}^{-1} \cdot khxh^{-1}nk^{-1} (t_1\dots t_n)^{-1}\big) \cdot f_1(t_1) \dots f_{n+1}(t_{n+1}). 
\end{aligned}
\end{equation}

\begin{lemma}
\label{change v lem}
\begin{equation}
\label{cycole equ-7}
\begin{aligned}
&\Phi_{x} \big(f_{n+1} \ast f_0, f_1 , \dots,  f_{n}\big)\\
=&\int_{ M/Z_M(x)}\int_{K N} \int_{G^{\times (n+1)}} C(\kappa(t_{n+1}k), t_1t_2\dots t_n, \dots,t_{n-1}t_n, t_n)  \\
&\cdot f_0 \big (khxh^{-1}nk^{-1} (t_1\dots t_{n+1})^{-1}\big) \cdot f_1(t_1) \dots f_{n+1}(t_{n+1}). 
\end{aligned}
\end{equation}
\begin{proof}
Using $G = KMAN$, we decompose  
\[
t_{n+1}^{-1} \cdot k = k_1 m_1 a_1 n_1 \in KMAN. 
\]
It follows that $k = t_{n+1} k_1 m_1 a_1n_1$ and $k = \kappa(t_{n+1} k_1)$.  We see 
\[
\begin{aligned}
&\Phi_{x} \big(f_{n+1} \ast f_0, f_1 , \dots,  f_{n}\big) =\int_{ M/Z_M(x)}\int_{K N} \int_{G^{\times (n+1)}} C(k, t_1t_2\dots t_n, \dots, t_{n-1}t_n, t_n)  \\
&f_0 \big (  k_1 m_1 a_1 n_1 hxh^{-1}nn_1^{-1}a_1^{-1}m_1^{-1} k_1^{-1} t_{n+1}^{-1} (t_1\dots t_n)^{-1}\big) \cdot f_1(t_1) \dots f_{n+1}(t_{n+1}). 
\end{aligned}
\]
Since $m_1 a_1 \in MA$ normalizes the nilpotent group $N$, we have the following identity,
\[
\begin{aligned}
&f_0 \big (  k_1 m_1 a_1 n_1 hxh^{-1}nn_1^{-1}a_1^{-1}m_1^{-1} k_1^{-1} t_{n+1}^{-1} (t_1\dots t_n)^{-1}\big)\\
=&f_0 \big (  k_1 m_1 hx(m_1h)^{-1}  \tilde{n}_1 n{n'}_1^{-1}k_1^{-1}  (t_1\dots t_{n+1})^{-1}\big),
\end{aligned}
\]
where $\tilde{n}_1, n'_1$ are defined by $n_1$ and $m_1,a_1, h,$ and $x$. 
Therefore,  renaming $\tilde{n}_1 n{n'}_1^{-1}$ by $n$, we have 
\[
\begin{aligned}
\Phi_{P,x} \big(f_{n+1} \ast f_0, f_1 , \dots,  f_{n}\big)=&\int_{ M/Z_M(x)}\int_{K N} \int_{G^{\times (n+1)}} C(\kappa(t_{n+1} k_1), t_1t_2\dots t_n, \dots, t_{n-1}t_n, t_n)  \\
&f_0 \big (  k_1  hxh^{-1}n k_1^{-1} (t_1\dots t_{n+1})^{-1}\big) \cdot f_1(t_1) \dots f_{n+1}(t_{n+1}).\\
\end{aligned}
\]
This completes the proof. 
\end{proof}
\end{lemma}

Combining (\ref{cycole equ-2}), (\ref{cycole equ-4}) and (\ref{cycole equ-7}), we have reached the following lemma.
\begin{lemma}
\label{coderiv lem_1}
The codifferential
\[
\begin{aligned}
&b\Phi_{P,x}^G(f_0, f_1, \dots, f_n, f_{n+1})\\
=& \int_{ M/Z_M(x)} \int_{K N} \int_{G^{\times (n+1)}} \tilde{C}(k, t_1, \dots t_{n+1})  \cdot f_0 \big (  k_1  hxh^{-1}n k_1^{-1} (t_1\dots t_{n+1})^{-1}\big) \\
&\qquad\qquad \qquad \qquad \qquad \cdot f_1(t_1) \dots f_{n+1}(t_{n+1}), 
\end{aligned}
\]
where $\tilde{C} \in C^\infty\big(K \times G^{\times n}\big)$ is given by 
\[
\begin{aligned}
 \tilde{C}(k, t_1, \dots t_{n+1}) =& \sum_{i=0}^{n} (-1)^i C(k, t_1t_2\dots t_{n+1}, \dots, (t_{i+1} \dots t_{n+1})^{\string^}, \dots, t_{n+1})\\
 +&(-1)^{n+1}C(\kappa(t_{n+1} k), t_1t_2\dots t_n, \dots, t_{n-1}t_n, t_n). 
 \end{aligned}
\]
\end{lemma}

\begin{lemma}
\label{coderiv lem_2}
\[
 \tilde{C}(k, t_1, \dots t_{n+1})  = 0. 
 \]
 \begin{proof}
To begin with,  we notice the following expression, 
\[
\begin{aligned}
&C\big(\kappa(t_{n+1} k), t_1t_2\dots t_n, \dots, t_{n-1}t_n, t_n\big)\\
=&\sum_{\tau \in S_n}\mathrm{sgn}(\tau)\cdot H_{\tau(1)}\big(t_1\dots t_n \kappa(t_{n+1} k)\big) H_{\tau(2)}\big(t_2\dots t_n  \kappa(t_{n+1} k) \big) \dots  H_{\tau(n)}\big(t_n \kappa(t_{n+1} k) \big). 
 \end{aligned}
\]
By Lemma \ref{function H}, we have 
\[
 H_{\tau(i)}\big(t_i\dots t_n  \kappa(t_{n+1} k) \big) =  H_{\tau(i)}\big(t_i\dots t_{n+1} k \big)  -  H_{\tau(i)}(t_{n+1} k). 
 \]
 Using the above property of $H_{\tau(i)}$, we have 
\[
\begin{aligned}
&C\big(\kappa(t_{n+1} k), t_1t_2\dots t_n, \dots, t_{n-1}t_n, t_n\big)\\
=&\sum_{\tau \in S_n}\mathrm{sgn}(\tau)\cdot  \Big(H_{\tau(1)}(t_1\dots t_{n+1} k) -H_{\tau(1)}(t_{n+1} k)   \Big) \\
&\cdot  \Big( H_{\tau(2)}(t_2\dots t_{n+1} k)  - H_{\tau(2)}(t_{n+1} k) \Big)\dots  \Big( H_{\tau(n)}(t_n t_{n+1} k) - H_{\tau(n)}(t_{n+1} k) \Big). \\
\end{aligned}
\]
By the symmetry of the permutation group $S_n$, in the above summation, the terms can contain at most one $H_{\tau(i)}(t_{n+1} k)$. Otherwise, it will be cancelled. Thus, 
\[
\begin{aligned}
&C\big(\kappa(t_{n+1} k), t_1t_2\dots t_n, \dots, t_{n-1}t_n, t_n\big)\\
=&\sum_{i=1}^n\sum_{\tau \in S_n}\mathrm{sgn}(\tau)\cdot  H_{\tau(1)}(t_1\dots t_{n+1} k) \dots  \Big( - H_{\tau(i)}(t_{n+1} k) \Big) \dots  H_{\tau(n)}(t_n t_{n+1} k) \\
+& \sum_{\tau \in S_n}\mathrm{sgn}(\tau)\cdot H_{\tau(1)}(t_1\dots t_{n+1} k)\dots H_{\tau(n)}(t_n t_{n+1} k). \\
\end{aligned}
\]
In the above expression, by changing the permutations 
\[
\big(\tau(1), \dots, \tau(n)\big ) \mapsto \big(\tau(1), \dots , \tau(i-1), \tau(n),  \tau(i), \dots ,\tau(n-1) \big),
\]
we get
\[
\begin{aligned}
&\sum_{\tau \in S_n}\mathrm{sgn}(\tau)\cdot  H_{\tau(1)}(t_1\dots t_{n+1} k) \dots H_{\tau(i-1)}(t_{i-1} \dots t_{n+1} k) \\
& \Big( - H_{\tau(i)}(t_{n+1} k) \Big) H_{\tau(i+1)}(t_{i+1} \dots t_{n+1} k) \dots  H_{\tau(n)}(t_n t_{n+1} k)\\
=&(-1)^{n- i}\sum_{\tau \in S_n}\mathrm{sgn}(\tau)\cdot  H_{\tau(1)}(t_1\dots t_{n+1} k) \dots  H_{\tau(i-1)}(t_{i-1}\dots t_n k) \\
& H_{\tau(n)}(t_{n+1} k)  H_{\tau(i)}(t_{i+1} \dots t_{n+1} k) \dots  H_{\tau(n-1)}(t_n t_{n+1} k).
\end{aligned}
\]
Putting all the above together, we have 
\[
\begin{aligned}
&(-1)^{n+1}C\big(\kappa(t_{n+1} k), t_1t_2\dots t_n, \dots, t_{n-1}t_n, t_n\big)\\
=&\sum_{i=0}^{n} (-1)^{i+1} \cdot C(k, t_1t_2\dots t_{n+1}, \dots, (t_{i+1} \dots t_{n+1})^{\string^}, \dots, t_{n+1}), 
\end{aligned}
\]
and 
\[
\begin{aligned}
 \tilde{C}(k, t_1, \dots t_{n+1}) =& \sum_{i=0}^{n} (-1)^i C(k, t_1t_2\dots t_{n+1}, \dots, (t_{i+1} \dots t_{n+1})^{\string^}, \dots, t_{n+1})\\
 +&(-1)^{n+1}C(\kappa(t_{n+1} k), t_1t_2\dots t_n, \dots, t_{n-1}t_n, t_n) = 0. 
 \end{aligned}
\]
 \end{proof}
\end{lemma}
We conclude from Lemma \ref{coderiv lem_1} and Lemma \ref{coderiv lem_2} that $\Phi_{x}$ is a Hochschild cocycle.

\subsection{Cyclic condition} \label{subsec:cyclic}
In this subsection, we prove that the cocycle $\Phi_x$ is cyclic.  Recall  
\[
\begin{aligned}
\Phi_{x}(f_1, \dots, f_n, f_0) = &  \int_{ M/Z_M(x)}\int_{K N} \int_{G^{\times n}} C(k, g_1g_2\dots g_n,  \dots, g_{n-1}g_n, g_n)  \\
&f_1 \big (khxh^{-1}nk^{-1} (g_1\dots g_n)^{-1}\big)f_2(g_1) \dots f_n(g_{n-1}) f_0(g_n). 
\end{aligned}
\]
By changing the variables, 
\[
t_1 =khxh^{-1}nk^{-1} (g_1\dots g_n)^{-1}, \hspace{5mm} t_j = g_{j-1}, \hspace{5mm} j = 2, \dots n-1, 
\]
we have 
\[
g_n = (t_1 \dots t_n)^{-1} khxh^{-1}nk^{-1}, 
\]
and 
\[
g_i \dots g_n = (t_1 \dots t_i)^{-1} khxh^{-1}nk^{-1}. 
\]
It follows that 
\[
\begin{aligned}
&\Phi_{x}(f_1, \dots, f_n, f_0)\\
=&   \int_{ M/Z_M(x)}\int_{K N} \int_{G^{\times n}} C\big(k, t_1^{-1} khxh^{-1}nk^{-1},  \dots, (t_1 \dots t_{n-1})^{-1} khxh^{-1}nk^{-1}, \\
&(t_1 \dots t_n)^{-1} khxh^{-1}nk^{-1} \big) \cdot f_0\big((t_1 \dots t_n)^{-1} khxh^{-1}nk^{-1}\big) f_1(t_1) f_2(t_2) \dots f_n(t_{n})\\
=&  \sum_{\tau \in S_n}\mathrm{sgn}(\tau)\cdot  \int_{ M/Z_M(x)} \int_{K N} \int_{G^{\times n}}  H_{\tau(1)}( t_1^{-1} khxh^{-1}n) \dots  H_{\tau(n)}((t_1 \dots t_n)^{-1} khxh^{-1}n)\\
&f_0\big((t_1 \dots t_n)^{-1} khxh^{-1}nk^{-1}\big) f_1(t_1) f_2(t_2) \dots f_n(t_{n}).\\
\end{aligned}
\]
We write 
\[
(t_1 \dots t_n)^{-1} k = k_1 m_1 a_1 n_1 \in KMAN.
\]
Then 
\begin{equation}
\label{k decomp-2}
k = (t_1 \dots t_n) k_1 m_1 a_1 n_1,
\end{equation}
and 
\[
\begin{aligned}
f_0\big((t_1 \dots t_n)^{-1} khxh^{-1}nk^{-1}\big)  = &f_0\big(k_1 m_1 a_1 n_1hxh^{-1} n n_1^{-1}a_1^{-1}m_1^{-1}k_1^{-1}(t_1 \dots t_n)^{-1} \big)\\
=&f_0\big(k_1h' x h'^{-1}n' k_1^{-1}(t_1 \dots t_n)^{-1} \big).
\end{aligned}
\]
Thus, 
\[
\begin{aligned}
\Phi_{x}(f_1, \dots, f_n, f_0)=&  \sum_{\tau \in S_n}\mathrm{sgn}(\tau)\cdot \int_{M/Z_M(x)} \int_{K N} \int_{G^{\times n}}  H_{\tau(1)}( t_1^{-1} k)  \dots  H_{\tau(n)}((t_1 \dots t_n)^{-1} k)\\
&\cdot f_0\big( k_1hxh^{-1}nk_1^{-1}(t_1 \dots t_n)^{-1}\big) f_1(t_1) f_2(t_2) \dots f_n(t_{n}). 
\end{aligned}
\]
By Lemma \ref{function H} and (\ref{k decomp-2}), for $1\leq i \leq n-1$, we have
\[
\begin{aligned}
 H_{\tau(i)}\big( (t_1\dots t_i)^{-1} k\big)=&-H_{\tau(i)}\big( t_1\dots t_i \kappa((t_1\dots t_i)^{-1} k)\big)\\
 =&-H_{\tau(i)}\big( t_1\dots t_i \kappa(t_{i+1} \dots t_n k_1)\big)\\
 =& H_{\tau(i)}\big( t_{i+1} \dots t_n k_1 \big) - H_{\tau(i)}\big( t_1 \dots t_n k_1)\big)
\end{aligned}
\]
and 
\[
\begin{aligned}
 H_{\tau(n)}\big( (t_1\dots t_n)^{-1} k\big)=&-H_{\tau(n)}\big( t_1\dots t_n \kappa((t_1\dots t_n)^{-1} k)\big)\\
 =&-H_{\tau(n)}\big( t_1\dots t_n k_1\big). 
\end{aligned}
\]
It follows that 
\[
\begin{aligned}
&\Phi_{x}(f_1, \dots, f_n, f_0)\\
=&  \sum_{\tau \in S_n}\mathrm{sgn}(\tau)\cdot\int_{ M/Z_M(m)} \int_{K N} \int_{G^{\times n}}  \prod_{i=1}^{n-1} \Big( H_{\tau(i)}\big( t_{i+1} \dots t_n k_1 \big) - H_{\tau(i)}\big( t_1 \dots t_n k_1)\big) \Big)\\
&\big(-H_{\tau(n)}( t_1\dots t_n k_1)\big)\cdot  f_0\big( k_1hxh^{-1}nk_1^{-1}(t_1 \dots t_n)^{-1}\big) f_1(t_1) f_2(t_2) \dots f_n(t_{n}).\\
\end{aligned}
\]
Note that 
\begin{equation}
\label{permutation}
\begin{aligned}
&\sum_{\tau \in S_n}\mathrm{sgn}(\tau)\cdot  \prod_{i=1}^{n-1} \Big( H_{\tau(i)}\big( t_{i+1} \dots t_n k_1 \big) - H_{\tau(i)}\big( t_1 \dots t_n k_1)\big) \Big) \cdot H_{\tau(n)}( t_1\dots t_n k_1)\\
=&\sum_{\tau \in S_n}\mathrm{sgn}(\tau)\cdot  \prod_{i=1}^{n-1} H_{\tau(i)}\big( t_{i+1} \dots t_n k_1 \big) \cdot H_{\tau(n)}( t_1\dots t_n k_1).
\end{aligned}
 \end{equation}
 In the above expression, by changing the permutation 
\[
\big(\tau(1), \dots, \tau(n)\big ) \mapsto \big(\tau(2), \dots ,\tau(n), \tau(1) \big) , 
\]
we can simplify Equation (\ref{permutation}) to the following one, 
\[
\begin{aligned}
&\sum_{\tau \in S_n}\mathrm{sgn}(\tau)\cdot  \prod_{i=1}^{n-1} \Big( H_{\tau(i)}\big( t_{i+1} \dots t_n k_1 \big) - H_{\tau(i)}\big( t_1 \dots t_n k_1)\big) \Big) \cdot H_{\tau(n)}( t_1\dots t_n k_1)\\
=&(-1)^{n-1} \cdot \sum_{\tau \in S_n}\mathrm{sgn}(\tau)\cdot  \prod_{i=1}^{n} H_{\tau(i)}\big( t_{i} \dots t_n k_1 \big). 
\end{aligned}
\]
Therefore, we have obtained the following identity, 
\[
\begin{aligned}
\Phi_{x}(f_1, \dots, f_n, f_0)=&(-1)^n\cdot  \sum_{\tau \in S_n}\mathrm{sgn}(\tau)\cdot \int_{ M/Z_M(x)}\int_{K N} \int_{G^{\times n}}  \prod_{i=1}^n H_{\tau(i)}\big( t_{i} \dots t_n k \big) \\
&\cdot f_0\big( khxh^{-1}nk^{-1}(t_1 \dots t_n)^{-1}\big) f_1(t_1) f_2(t_2) \dots f_n(t_{n})\\
=&(-1)^n \cdot \Phi_{P,x}(f_0, \dots, f_n). 
\end{aligned}
\]

Hence, we conclude that $\Phi_x$ is a cyclic cocycle, and have completed the proof of Theorem \ref{thm cocycle}.

\section{The Fourier transform of  $\Phi_x$}
\label{sec:fourier-transform}

In this section, we study the Fourier transform of the cyclic cocycle $\Phi_x$  introduced in Section \ref{sec:cyclic-cocycle}. For the convenience of readers, we start with recalling the basic knowledge about parabolic induction and the Plancherel formula in Section \ref{subsec:parabolic-ind} and \ref{subsec:wave-packet}. 

\subsection{Parabolic induction}\label{subsec:parabolic-ind}

A brief introduction to discrete series representations can be found in Appendix \ref{app:discrete}. In this section, we review the construction of parabolic induction. Let $H$ be a $\theta$-stable Cartan subgroup of $G$ with Lie algebra $\kh$. Then $\kh$ and $H$ have the following decompositions, 
\[
\kh = \kh_k + \kh_p, \hspace{5mm} \kh_k = \kh \cap \kk,\qquad \kh_p = \kh \cap \kp,
\]
and $H_K = H \cap K$, $H_p = \exp(\kh_p)$. Let $P$ be a parabolic subgroup of $G$ with the split part $H_p$, that is
\[
P = M_P H_p N_P = M_P A_P N_P. 
\]

\begin{definition}
Let $\sigma$ be a (limit of) discrete series representation of $M_P$ and $\varphi$ a unitary character of $A_P$. The product $\sigma \otimes \varphi$ defines a unitary representation of the Levi subgroup $L=M_PA_P$. A \emph{basic representation} of $G$ is a representation by extending  $\sigma \otimes \varphi$ to $P$ trivially across $N_P$ then inducing to $G$:
\[
\pi_{\sigma, \varphi} = \Ind_P^G(\sigma \otimes \varphi). 
\]
If $\sigma$ is a discrete series then $\Ind_P^G(\sigma \otimes \varphi)$ will be called a \emph{basic representation induced from discrete series}. This is also known as \emph{parabolic induction}. 
\end{definition}

The character of $\pi_{\sigma, \varphi}$ is given in Theorem\ \ref{thm:B.3}, Equation (\ref{induced character eq}), 
and Corollary \ref{coro character}. Note that basic representations might not be irreducible. Knapp and Zuckerman complete the classification of tempered representations\footnote{Knapp and Zuckerman prove that every tempered representation of $G$ is basic, and every basic representation is tempered.} by showing which basic representations are irreducible.

Now consider a single parabolic subgroup $P \subseteq G$ with associated Levi subgroup $L$, and form the group
\[
W(A_P, G) = N_K(\ka_P)/Z_K(\ka_P), 
\]
where $N_K(\ka_P)$ and $Z_K(\ka_P)$ are the normalizer and centralizer of $\ka_P$ in $K$ respectively. The group $W(A_P, G)$ acts as outer automorphism of $M_P$, and also the set of equivalence classes of representations of $M_P$. For any discrete series representation $\sigma$ of $M_P$,  we define
\[
W_\sigma = \big\{ w \in N_K(\ka_P)\colon \Ad_w^* \sigma \cong \sigma \big\} / Z_K(\ka_P). 
\]
Then the above Weyl group acts on the family of induced representations
\[
\big\{ \Ind_P^G(\sigma \otimes \varphi)\big\}_{\varphi \in \widehat{A}_P}. 
\]

\begin{definition}
Let $P_1$ and $P_2$ be two  parabolic subgroups of $G$ with Levi factors $L_i = M_{P_i} A_{P_i}$. Let $\sigma_1$ and $\sigma_2$ be two discrete series representations of $M_{P_i}$. We say that 
\[
(P_1, \sigma_1) \sim (P_2, \sigma_2)
\] 
if there exists an element $w$ in $G$ that conjugates the Levi factor of $P_1$ to the Levi factor of $P_2$, and conjugates $\sigma_1$ to a representation unitarily equivalent to $\sigma_2$. In this case,  there is a unitary $G$-equivariant isomorphism 
\[
\Ind_{P_1}^G(\sigma_1 \otimes \varphi) \cong \Ind_{P_2}^G(\sigma_2 \otimes (\Ad_w^*\varphi))
\]
that covers the isomorphism 
\[
\Ad_w^*  \colon C_0(\widehat{A}_{P_1}) \to  C_0(\widehat{A}_{P_2}). 
\]
 We denote by $[P,\sigma]$ the equivalence class of $(P, \sigma)$, and $\mathcal{P}(G)$ the set of all equivalence classes. 
\end{definition}


At last, we recall the functoriality of parabolic induction. 
\begin{lemma}
 If $S = M_S A_S N_S$ is any  parabolic subgroup of $L$, then the unipotent radical of $SN_P$ is $N_SN_P$, and the product 
 \[
 Q = M_Q A_Q N_Q = M_S (A_SA_P) (N_SN_P)
 \] 
 is a parabolic subgroup of $G$.
 \begin{proof}
See \cite[Lemma 4.1.1]{Voganbook}. 
 \end{proof}
\end{lemma}

\begin{theorem}[Induction in stage]
\label{induction in stage}
 Let $\eta$ be a unitary representation (not necessarily a discrete series representation) of $M_S$. We decompose 
\[
\varphi= (\varphi_1, \varphi_2) \in \widehat{A}_{S} \times \widehat{A}_P.
\]
There is a canonical equivalence 
\[
\Ind_{P}^G\big(\Ind^{M_P}_{S}(\eta \otimes \varphi_1) \otimes \varphi_2  \big)\cong  \Ind_{Q}^G\big(\eta \otimes (\varphi_1, \varphi_2)\big). 
\]
\begin{proof}
 See \cite[P. 170]{MR1880691}.
\end{proof}
\end{theorem}
\subsection{Wave packet}\label{subsec:wave-packet}
Let $G$ be a connected, linear,  real reductive Lie group as before, and $\widehat{G}_\mathrm{temp}$ be the set of equivalence classes of irreducible unitary tempered representations of $G$. For a Schwartz function $f$ on $G$, its Fourier transform $\widehat{f}$  is defined by 
\[
\widehat{f}(\pi) = \int_G f(g) \pi(g) dg, \hspace{5mm} \pi \in \widehat{G}_\mathrm{temp}. 
\]
Thus, the Fourier transform assigns to $f$ a family of operators on different Hilbert spaces $\pi$. 

The group $A_P$, which consists entirely of positive definite matrices, is isomorphic to its Lie algebra via the exponential map.
So $A_P$ carries the structure of a vector space, and we can speak of its space of Schwartz functions in the ordinary sense of harmonic analysis. The same
goes for the unitary (Pontryagin) dual $\widehat{A}_P$. By a tempered measure on $A_P$ we
mean a smooth measure for which integration extends to a continuous linear
functional on the Schwartz space. Recall Harish-Chandra's Plancherel formula for $G$, c.f. \cite{MR0439994}.
\begin{theorem}
There is a unique smooth, tempered, $W_\sigma$-invariant function $m_{P, \sigma}$ on the spaces $\widehat{A}_P$ such that
\[
\|f\|_{L^{2}(G)}^{2}=\sum_{[P, \sigma]} \int_{\widehat{A}_{P}}\left\|\widehat{f}(\pi_{\sigma, \varphi})\right\|_{HS}^{2} m_{P, \sigma}(\varphi) d\varphi
\]
for every Schwartz function $f \in \mathcal{S}(G)$. We call $m_{P, \sigma}(\varphi)$ the Plancherel density of  the representation $\Ind_P^G(\sigma \otimes \varphi)$.  
\end{theorem}

As $\varphi \in \widehat{A}_P$ varies, the Hilbert spaces 
\[
\pi_{\sigma, \varphi} = \Ind_P^G(\sigma \otimes \varphi)
\]
 can be identified with one another as representations of $K$. Denote by $\Ind_P^G \sigma$ this common Hilbert space, and  $\mathcal{L}(\Ind_P^G \sigma)$ the space of $K$-finite Hilbert-Schmidt operators on $\Ind_P^G \sigma$. We shall discuss the adjoint to the Fourier transform. 

\begin{definition}
Let $h$ be a Schwarz-class function from $\widehat{A}_P$ into  operators on $\Ind_P^G \sigma$ such that it is invariant under the $W_\sigma$-action. That is 
\[
h \in \big[L^2(\widehat{A}_P) \otimes \mathcal{L}^2(\Ind_P^G \sigma)  \big]^{W_\sigma}.
\]
 The \emph{wave packet} associated to $h$ is the scalar function defined by the following formula
\[
\check{h}(g) = \int_{\widehat{A}_P} \tr \big(\pi_{\sigma, \varphi}(g^{-1}) \cdot h(\varphi)  \big) \cdot m_{P, \sigma}(\varphi) d\varphi. 
\]
\end{definition}

A fundamental theorem of Harish-Chandra asserts that wave packets are Schwartz functions on $G$. 
\begin{theorem}
The wave packets associated to the Schwartz-class functions from $\widehat{A}_P$ into $\mathcal{L}(\Ind_P^G \sigma)$ all belong to the Harish-Chandra Schwarz space $\mathcal{S}(G)$. Moreover, 
The wave packet operator $h \to \check{h}$ is adjoint to the Fourier transform.
\begin{proof}
See  \cite[Theorems 12.7.1 and 13.4.1] {MR1170566} and \cite[Corollary 9.8]{MR3518312}.
\end{proof}
\end{theorem}

\subsection{Derivatives of Fourier transform}
Let $P = MAN$ be a parabolic subgroup. Suppose that $\pi = \ind_P^G(\eta^{M} \otimes \varphi)$ where $\eta^{M}$ is an irreducible tempered representation of $M$ with character denoted by $\Theta^M(\eta)$ and 
\[
\varphi \in \widehat{A}_P = \widehat{A}_\cc \times \widehat{A}_S.
\] 
We denote $r = \dim(\widehat{A}_S)$ and $n = \dim(\widehat{A}_\cc)$. For $i = 0, \dots, n$, let $h_i \in \mathcal{S}(\widehat{A}_P)$, and $v_i, w_i$ be unit $K$-finite vectors in $\Ind_P^G(\eta^M)$.

\begin{definition}
 If  $f_i\in \mathcal{S}(G)$ are wave packets associated to 
\[
 h_i \cdot v_{i} \otimes w_i^* \in \calS\big(\widehat{A}_P, \calL(\Ind_P^G(\eta^M)\big), 
\]
then we define a $(n+1)$-linear map $T_\pi$ with image in $\mathcal{S}(\widehat{A}_P)$ by 
\begin{equation}
\label{Tpi equ}
\begin{aligned}
&T_\pi(\widehat{f}_0, \dots,  \widehat{f}_n) \\
=&
\begin{cases} 
       \sum_{\tau \in S_n}\mathrm{sgn}(\tau) \cdot h_0(\varphi) \cdot \prod_{i=1}^n \frac{\partial h_i(\varphi)}{\partial_{\tau(i)}}  & \text{if } \  v_i = w_{i+1}, i = 0, \dots, n-1, \ \text{and} \ v_n = w_0;\\
      0 & \text{else}. \\
   \end{cases}
   \end{aligned}
\end{equation}
\end{definition}

Next we want to generalize the above definition to the Fourier transforms of all $f \in \mathcal{S}(G)$.  The induced space $\pi = \Ind_P^G(\eta^M \otimes \varphi)$ has a dense subspace:
\begin{equation}
\label{induced space}
\big\{s \colon K \to  V^{\eta^M} \ \mathrm{continuous} \big| s(km) = \eta^M(m)^{-1}s(k) \ \mathrm{for} \ k \in K, m \in K \cap M \big\},
\end{equation}
where $V^{\eta^M}$ is the Hilbert space of $M$-representation $\eta^M$. The group action on $\pi$ is given by the formula
\begin{equation}
\label{g action}
\pi(g) s(k) = e^{-\langle \log \varphi + \rho,  H(g^{-1}k)\rangle } \cdot \eta^M(\mu(g^{-1}k))^{-1} \cdot s(\kappa(g^{-1}k)),
\end{equation}
where $\rho = \frac{1}{2}\sum_{\alpha \in \mathcal{R}^+(\ka_P, \kg)}\alpha$. By Equation (\ref{g action}), the Fourier transform
\[
\begin{aligned}
\pi(f) s(k) =& \widehat{f}(\pi)s(k) \\
=& \int_G(e^{-\langle \log \varphi + \rho,  H(g^{-1}k)\rangle } \cdot \eta^M(\mu(g^{-1}k))^{-1} f(g) \cdot  s(\kappa(g^{-1}k)) dg. 
\end{aligned}
\]

Suppose now that $f_0, \dots, f_n $ are arbitrary Schwartz functions on $G$ and $\widehat{f}_0,\dots, \widehat{f}_n$ are their Fourier transforms. 
\begin{definition}
\label{defn:T-pi}
 For any $1 \leq i \leq n$, we define  a linear operator $\frac{\partial}{\partial_i}$ on $\pi = \Ind_P^G(\eta^M \otimes \varphi)$ by the following formula: 
\begin{equation}
\label{deritive action}
\Big(\frac{\partial\widehat{f}(\pi)}{\partial_i}\Big)s(k): = \int_G H_i(g^{-1}k) \cdot (e^{-\langle \log \varphi + \rho,  H(g^{-1}k)\rangle } \cdot \eta^M(\mu(g^{-1}k))^{-1} \cdot f(g) \cdot s(\kappa(g^{-1}k)). 
\end{equation}	
We define a $(n+1)$-linear map $T: \mathcal{S}(G)\to \mathcal{S}(\widehat{A}_P)$ by 
\[
T_{\pi}(\widehat{f}_0, \dots \widehat{f}_n) :=\sum_{\tau \in S_n}\mathrm{sgn}(\tau) \cdot \tr \Big( \widehat{f}_0(\pi) \cdot \prod_{i=1}^n \frac{\partial \widehat{f}_i(\pi)}{\partial_{\tau(i)}} \Big). 
\]
The above definition generalizes (\ref{Tpi equ}). 
\end{definition}

\begin{proposition}
\label{key prop}
For any $\pi = \Ind_P^G(\eta^M \otimes \varphi)$, we have the following identity:
\begin{equation}
\label{trace iden}
\begin{aligned}
&T_{\pi}(\widehat{f}_0, \dots \widehat{f}_n)  \\
=&(-1)^n\sum_{\tau \in S_n}\mathrm{sgn}(\tau)\int_{KMAN}\int_{G^{\times n} } H_{\tau(1)}\big(g_1 \dots g_n k\big)\dots H_{\tau(n)}\big(g_nk\big)\\
 & e^{\langle \log \varphi + \rho,  \log a\rangle } \cdot \Theta^M(\eta)(m) \cdot f_0(kmank^{-1} (g_1 g_2 \dots g_n)^{-1}) f_1(g_1) \dots f_n(g_n).
\end{aligned}
\end{equation}
\begin{proof}
By definition, for any $\tau \in S_n$, 
\[
\begin{aligned}
&\Big( \widehat{f}_0(\pi) \cdot \prod_{i=1}^k \frac{\partial \widehat{f}_i(\pi)}{\partial_{\tau(i)}} \Big)s(k)\\
 =& \int_{G^{\times {(k+1)}} } H_{\tau(1)}(g_1^{-1}\kappa(g_0^{-1}k)) H_{\tau(2)}\big(g_2^{-1}\kappa((g_0g_1)^{-1}k)\big)\\
& H_{\tau(n)}\big(g_n^{-1}\kappa((g_0g_1\dots g_{n-1})^{-1}k)\big)\cdot  e^{-\langle \log \varphi+\rho, H((g_0g_1\dots g_n)^{-1}k) \rangle } \\
&\eta^M(\mu((g_0g_1\dots g_n)^{-1}k))^{-1} \cdot f_0(g_0) f_1(g_1) \dots f_n(g_n) s\big(\kappa((g_0g_1\dots g_n)^{-1}k)\big).  
\end{aligned}
\]
By setting $g = (g_0g_1\dots g_n)^{-1}k$, we have 
\[
g_0= kg^{-1} (g_1 g_2 \dots g_n)^{-1}, 
\] 
and 
\[
(g_0 g_1 \dots g_j)^{-1} k = g_{j+1} g_{j+2} \dots g_n g. 
\]
We denote $g^{-1} = mank'^{-1} \in MA NK = G$. Thus, 
\begin{equation}
\label{T equ_1}
\begin{aligned}
&\left( \widehat{f}_0(\pi) \cdot \prod_{i=1}^k \frac{\partial \widehat{f}_i(\pi)}{\partial_{\tau(i)}} \right)s(k)\\
=&\int_{KMAN} \int_{G^{\times k} } H_{\tau(1)}\left(g_1^{-1}\kappa(g_1 \dots g_n k)\right) \dots H_{\tau(n)}\left(g_n^{-1}\kappa(g_nk)\right)\\
& e^{\langle \log \varphi+\rho, \log a\rangle} \cdot \eta^M(m) \cdot f_0\left(kmank'^{-1} \left(g_1 g_2 \dots g_n\right)^{-1}\right) f_1(g_1) \dots f_n(g_n) s(k').  
\end{aligned}
\end{equation}
By Lemma \ref{function H}, 
\[
H_{\tau(i)}\big(g_i^{-1}\kappa(g_i \dots g_n k)\big) = H_{\tau(i)}\big(g_{i+1} \dots g_n k\big) - H_{\tau(i)}\big(g_{i} \dots g_n k\big).
\]
Thus, 
\begin{equation}
\label{product H}
\begin{aligned}
&\sum_{\tau \in S_n}\mathrm{sgn}(\tau) H_{\tau(1)}\big(g_1^{-1}\kappa(g_1 \dots g_n k)\big) H_{\tau(2)}\big(g_2^{-1}\kappa((g_2g_3\dots g_nk)\big)\dots  H_{\tau(n)}\big(g_n^{-1}\kappa(g_nk)\big)\\
=&\sum_{\tau \in S_n}\mathrm{sgn}(\tau) \big(H_{\tau(1)}(g_2 \dots g_n k) - H_{\tau(1)}(g_1g_2 \dots g_n k)\big)\big(H_{\tau(2)}(g_3 \dots g_n k) - H_{\tau(2)}(g_2 \dots g_n k)\big) \\
& \dots \big(H_{\tau(n-1)}(g_n k) - H_{\tau(n-1)}(g_{n-1}g_n k)\big) \big(-  H_{\tau(n)}(g_nk) \big).\\
 \end{aligned}
\end{equation}
By induction on $n$, we can prove that the right hand side of equation (\ref{product H}) equals to 
\[
(-1)^n\sum_{\tau \in S_n}\mathrm{sgn}(\tau) H_{\tau(1)}\big(g_1 \dots g_n k\big) H_{\tau(2)}\big(g_2g_3\dots g_nk\big)\cdot  H_{\tau(n)}\big(g_nk\big).
\]
By  (\ref{T equ_1}) and (\ref{product H}), we conclude that 
\[
\begin{aligned}
&\sum_{\tau \in S_n}\mathrm{sgn}(\tau) \Big( \widehat{f}_0(\pi) \cdot \prod_{i=1}^k \frac{\partial \widehat{f}_i(\pi)}{\partial_{\tau(i)}} \Big)s(k)\\
=&(-1)^n \sum_{\tau \in S_n}\mathrm{sgn}(\tau) \int_{KMAN} \int_{G^{\times k} }H_{\tau(1)}\big(g_1 \dots g_n k\big) H_{\tau(2)}\big(g_2g_3\dots g_nk\big)\cdot  H_{\tau(n)}\big(g_nk\big)\\
& e^{\langle \log \varphi+\rho, \log a \rangle} \cdot \eta^M(m) \cdot f_0\big(kmank'^{-1} (g_1 g_2 \dots g_n)^{-1}\big) f_1(g_1) \dots f_n(g_n) s(k').  
\end{aligned}
\]
Expressing  it as a kernel operator, we have
\[
 \sum_{\tau \in S_n}\mathrm{sgn}(\tau) \cdot  \Big( \widehat{f}_0(\pi) \cdot \prod_{i=1}^k \frac{\partial \widehat{f}_i(\pi)}{\partial_{\tau(i)}} \Big)s(k) = \int_{K} L(k, k') s(k') dk'. 
\]
where
\[
\begin{aligned}
L(k, k') =& (-1)^n\sum_{\tau \in S_n}\mathrm{sgn}(\tau)\int_{MAN}\int_{G^{\times k} } H_{\tau(1)}\big(g_1 \dots g_n k\big) H_{\tau(2)}\big(g_2g_3\dots g_nk\big)\cdot  H_{\tau(n)}\big(g_nk\big)\\
& e^{\langle \log \varphi+\rho, \log a \rangle} \cdot \eta^M(m) \cdot  f_0(kmank'^{-1} (g_1 g_2 \dots g_n)^{-1}) f_1(g_1) \dots f_n(g_n). 
\end{aligned}
\]
The proposition follows from the fact that  $T_{\pi} = \int_K L(k, k) dk$. 
\end{proof}
\end{proposition}

Suppose  that $P_1$ and $P_2$ are two  non-conjugated parabolic subgroups  such that $\Ind_{P_2}^G(\sigma_2 \otimes \varphi_2)$ can be embedded into $\Ind_{P_1}^G(\sigma_1 \otimes \varphi_1)$, i.e. \
\[
\Ind_{P_1}^G\big(\sigma_1 \otimes \varphi_1\big) = \bigoplus_k \Ind_{P_2}^G\big(\delta_k \otimes \varphi_2\big), 
\]
 where $\sigma_1$ is a discrete series representation of $M_1$ and $\delta_k$ are different limit of discrete series representations of $M_2$. We decompose
 \[
 \widehat{A}_1 = \widehat{A}_2  \times \widehat{A}_{12} = \widehat{A}_\cc \times \widehat{A}_S \times \widehat{A}_{12}.
 \]
Let $h_i \in \mathcal{S}(\widehat{A}_1 )$, and $v_i, w_i$ be unit $K$-finite vectors in $\Ind_{P_1}^G(\sigma_1)$ for $i = 0, \dots n$. We put
\[
\widehat{f}_i = h_i \cdot v_{i} \otimes w_i^* \in \calS\big(\widehat{A}_1, \calL(\Ind_{P_1}^G(\sigma_1)\big).
\]
The following lemma follows from Definition \ref{defn:T-pi}. 
\begin{lemma}
\label{Tpi lem-2}
Suppose that $\pi = 	\Ind_{P_2}^G\big(\delta_k \otimes \varphi_2\big)$. If 
\[
v_i = w_{i+1}, \quad i= 0, \dots, n-1,
\]
and $v_n = w_0 \in \pi$, then 
\[
T_\pi(\widehat{f}_0 , \dots,  \widehat{f}_n) =  \sum_{\tau \in S_n}\mathrm{sgn}(\tau) \cdot h_0\big((\varphi_2, 0)\big) \cdot \prod_{i=1}^n \frac{\partial h_i\big((\varphi_2, 0)\big)}{\partial_{\tau(i)}}.
\]
Otherwise $T_\pi(\widehat{f}_0 , \dots, \widehat{f}_n) = 0$. 
\end{lemma}

\subsection{Cocycles on $\widehat{G}_\mathrm{temp}$}
Let $P_\cc = M_\cc A_\cc N_\cc$ be a maximal parabolic subgroup and $T$ be the maximal torus of $K$. In particular, $T$ is the compact Cartan subgroup of $M_\cc$. 

\begin{definition}
For an irreducible tempered representation $\pi$ of $G$, we define 
\[
\mathcal{A}(\pi) = \left\{\eta^{M_\cc}\otimes \varphi \in (\widehat{M_\cc A_\cc})_{\mathrm{temp}} \Big| \Ind^G_{P_\cc}(\eta^{M_\cc}\otimes \varphi) = \pi\right\}.
\]	
\end{definition}

\begin{definition}
Let $m(\eta^{M_\cc})$ be the Plancherel density for the irreducible tempered representations $\eta^{M_\cc}$ of $M_\cc$. We put 
\[
\mu\big(\pi \big)  =  \sum_{\eta^{M_\cc}\otimes \varphi \in \mathcal{A}(\pi)} m(\eta^{M_\cc}). 
\]
\end{definition}

Recall the  Plancherel formula 
\[
f(e) =\int_{\pi \in \widehat{G}_\mathrm{temp}} \tr\big(\widehat{f}(\pi) \big)\cdot  m(\pi)d\pi,
\]
where $m(\pi)$ is the Plancherel density for the $G$-representation $\pi$. 
\begin{definition}\label{defn:phi-hat}
We define $\widehat{\Phi}_{e}$ by the following formula:
\[
\widehat{\Phi}_{e}(\widehat{f}_0 , \dots, s \widehat{f}_n)  =  \int_{\pi \in \widehat{G}_\mathrm{temp}}     T_{\pi}(\widehat{f}_0, \dots,  \widehat{f}_n) \cdot  \mu(\pi) \cdot  d\pi.
\]
\end{definition}

\begin{theorem}
\label{main thm-2}
For any $f_0, \dots f_n \in \mathcal{S}(G)$, the following identity holds, 
\[
 \Phi_{e}(f_0, \dots, f_n) =(-1)^n\widehat{\Phi}_{e}(\widehat{f}_0, \dots, \widehat{f}_n).
\]
\end{theorem}
The proof of Theorem \ref{main thm-2} is presented in Section \ref{subsec:proof-e}.

\begin{example}\label{ex:Rn-cocycle}
Suppose that $G = \R^n$. Let 
 \[
x^i =  (x^i_1, \dots x^i_{n}) \in \R^n
 \]
be the coordinates of $\R^{n}$. On $\mathcal{S}(\mathbb{R}^n)$, the cocycle $\Phi_e$ is given as follows,
\[
\begin{aligned}
&\Phi_e(f_0, \dots f_{n})\\
 =& \sum_{\tau \in S_{n}} \mathrm{sgn}(\tau) \int_{x^1 \in \R^{n}} \dots \int_{x^{n}\in \R^{n}} x^1_{\tau(1)} \dots x^{n}_{\tau(n)}f_0\big(-(x^1 +\dots+ x^{n}) \big)f_1(x^1) \dots f_{n}(x^{n}).
\end{aligned}
\]
On the other hand, the cocycle $\widehat{\Phi}_e$ on $\mathcal{S}(\widehat{\mathbb{R}^n})$ is given as follows, 
\[
\widehat{\Phi}_e(\widehat{f}_0, \dots, \widehat{f}_{n}) =(-1)^n \int_{\R^{n}} \widehat{f}_0 d  \widehat{f}_1 \dots  d\widehat{f}_{n}. 
\]
To see they are equal, we can compute
\[
 \begin{aligned}
\Phi_e(f_0,  \dots, f_n) =& \sum_{\tau \in S_n}\mathrm{sgn}(\tau)\cdot \Big(f_0 \ast (x_{\tau(1)} f_1) \ast \dots \ast (x_{\tau(n)} f_n) \Big)(0). \\
=& \sum_{\tau \in S_n} \mathrm{sgn}(\tau)\cdot  \int_{\R^n}  \Big( \reallywidehat{f_0 \ast \big(x_{\tau(1)} f_1\big) \ast \dots \ast \big(x_{\tau(n)} f_n\big)} \Big)\\
=& \sum_{\tau \in S_n} \mathrm{sgn}(\tau)\cdot \int_{\R^n}   \Big( \widehat{f_0} \cdot \reallywidehat{\big(x_{\tau(1)} f_1\big)}   \dots  \reallywidehat{\big(x_{\tau(n)} f_n\big)}  \Big)\\
=&\sum_{\tau \in S_n}\mathrm{sgn}(\tau)\cdot \int_{\R^n }(-1)^n\Big( \widehat{f_0} \cdot \frac{\partial\widehat{f_1}}{\partial_{\tau(1)}}  \dots  \frac{\partial\widehat{f_n}}{\partial_{\tau(n)}}  \Big)\\
=&(-1)^n\int_{\R^n } \widehat{f}_0 d\widehat{f}_1 \dots d\widehat{f}_n.
\end{aligned}
\]
\end{example}

To introduce the cocycle $\widehat{\Phi}_t$ for any $t \in T^\text{reg}$, we first recall the formula of orbital integral (\ref{formula orbit}) splits into three parts:
\[
 \text{regular \ part} + \text{singular \ part} + \text{higher \ part}. 
\]
Accordingly, for any $t \in  T^{\text{reg}}$, we define

\begin{itemize}
\item 
regular part: for regular $\lambda$$ \in \Lambda^*_{K} + \rho_c$ (see Definition \ref{defn:regular-singular}), we define
\[
\begin{aligned}
&\left[\widehat{\Phi}_{t}(\widehat{f}_0, \dots, \widehat{f}_n)\right]_{\lambda}\\
=& \left(\sum_{w \in W_{K}} (-1)^w \cdot e^{w \cdot \lambda}(t) \right) \cdot \int_{\varphi \in \widehat{A}_\cc}      T_{\Ind_{P_\cc}^G(\sigma^{M_\cc}(\lambda)\otimes \varphi)}(\widehat{f}_0 , \dots,   \widehat{f}_n)\cdot  d\varphi, 	
\end{aligned}
\]
where $\sigma^{M_\cc}(\lambda)$ is the discrete series representation of $M_\cc$ with Harish-Chandra parameter $\lambda$. 
\item 
singular part: for any singular $\lambda \in \Lambda^*_{K} + \rho_c$, we define
\[
\begin{aligned}
&\left[\widehat{\Phi}_{t}(\widehat{f}_0, \dots, \widehat{f}_n)\right]_{\lambda}\\
=&  \frac{\sum_{w \in W_{K}} (-1)^w \cdot e^{w \cdot \lambda}(t) }{n(\lambda)} \cdot \sum_{i=1}^{n(\lambda)} \int_{\varphi \in \widehat{A}_\cc} \epsilon(i) \cdot T_{\Ind_{P_\cc}^G(\sigma^{M_\cc}_i(\lambda)\otimes \varphi)}(\widehat{f}_0 , \dots,   \widehat{f}_n)\cdot d\varphi\\
\end{aligned}
\]
where $\sigma^{M_\cc}_i$ are limit of discrete series representations of $M_\cc$ with Harish-Chandra parameter $\lambda$ organized as in Theorem \ref{orbital formula}, and $\epsilon(i) = 1$ for $i = 1, \dots \frac{n(\lambda)}{2}$ and $\epsilon(i) = -1$ for $i = \frac{n(\lambda)}{2}+1, \dots n(\lambda)$. 
\item
higher part:
\[
\begin{aligned}
&\left[\widehat{\Phi}_{t}(\widehat{f}_0, \dots, \widehat{f}_n)\right]_{\mathrm{high}}\\
=&\int_{\pi \in \widehat{G}^\text{high}_\mathrm{temp}} T_{\pi}(\widehat{f}_0, \dots, \widehat{f}_n)\cdot \left( \sum_{\eta^{M_\cc} \otimes \varphi \in \mathcal{A}(\pi)} \kappa^{M_\cc} (\eta^{M_\cc}, t)
\right) \cdot d\varphi,
\end{aligned}
\] 
where the functions $\kappa^{M_\cc}(\eta^{M_\cc}, t)$ are defined in Subsection \ref{section orbital}, and 
\[
 \widehat{G}^\text{high}_\mathrm{temp} = \big\{\pi \in   \widehat{G}_\mathrm{temp} \big| \pi =\Ind_{P_\cc}^G( \eta^{M_\cc} \otimes \varphi), \eta^{M_\cc} \in  (\widehat{M}_\cc)^\text{high}_\mathrm{temp}   \big\}.
 \]
\end{itemize}

\begin{definition}
For any element $t \in T^\text{reg}$, we define
\[
\begin{aligned}
&\widehat{\Phi}_{t}(\widehat{f}_0, \dots, \widehat{f}_n) \\
 =& \sum_{\text{regular} \ \lambda \in \Lambda^*_{K}+ \rho_c} \left[\widehat{\Phi}_{h}(\widehat{f}_0, \dots, \widehat{f}_n)\right]_{\lambda} + \sum_{\text{singular} \ \lambda \in \Lambda^*_{K}+ \rho_c} \left[\widehat{\Phi}_{t}(\widehat{f}_0, \dots, \widehat{f}_n)\right]_{\lambda}+  \left[\widehat{\Phi}_{t}(\widehat{f}_0, \dots, \widehat{f}_n)\right]_{\mathrm{high}}. 
\end{aligned} 
\]
\end{definition}

\begin{theorem}
\label{main thm-3}
For any $t \in T^\text{reg}$, and $f_0, \dots f_n \in \mathcal{S}(G)$, the following identity holds,
\[
\Delta_{T}^{M_\cc}(t)\cdot  \Phi_{t}(f_0, \dots, f_n) = (-1)^n\widehat{\Phi}_{t}(\widehat{f}_0, \dots, \widehat{f}_n).
\]
\end{theorem}

The proof of Theorem  \ref{main thm-3} is presented in Section \ref{subsec:proof-x}.

\subsection{Proof of Theorem \ref{main thm-2}}\label{subsec:proof-e}
We split the proof into several steps:\\

\noindent $\mathbf{Step \ 1} \colon$ Change the integral from  $\widehat{G}_\text{temp}$ to $(\widehat{M_\cc A_\cc} )_\text{temp}$:
\[
\begin{aligned}
&\widehat{\Phi}_{e}(f_0, \dots,  f_n)  \\
=& \int_{\pi \in \widehat{G}_\mathrm{temp}}     T_{\pi}(\widehat{f}_0, \dots,  \widehat{f}_n) \cdot \mu(\pi) \cdot d\pi \\
=& \int_{\eta^{M_\cc} \otimes \varphi \in (\widehat{M_\cc A_\cc} )_\text{temp}}     T_{\Ind_{P_\cc}^G(\eta^{M_\cc} \otimes \varphi)}(\widehat{f}_0 , \dots,  \widehat{f}_n)\cdot m(\eta^{M_\cc}).\\
\end{aligned} 
\]

\noindent $\mathbf{Step \ 2} \colon$ Replace $T_{\Ind_{P_\cc}^G(\eta^{M_\cc} \otimes \varphi)}$ in the above expression of $\widehat{\Phi}_e$ by Equation (\ref{trace iden}):
\[
\begin{aligned}
&\widehat{\Phi}_{e}(\widehat{f}_0 , \dots,  \widehat{f}_n)  \\
=&(-1)^n\sum_{\tau \in S_n}\mathrm{sgn}(\tau) \int_{\eta^{M_\cc}\otimes \varphi \in (\widehat{M_\cc A_\cc} )_\text{temp}}  \int_{KM_\cc A_\cc N_\cc}\int_{G^{\times n} } H_{\tau(1)}\big(g_1 \dots g_n k\big)\dots H_{\tau(n)}\big(g_nk\big)\\
 & e^{\langle \log \varphi+\rho,  \log a\rangle} \cdot \Theta^{M_\cc}(\eta)(m) \cdot f_0\left(kmank^{-1} (g_1 g_2 \dots g_n)^{-1}\right) f_1(g_1) \dots f_n(g_n) \cdot  m(\eta^{M_\cc}) \\
 \end{aligned}
 \]

\noindent $\mathbf{Step \ 3} \colon$ Simplify $\widehat{\Phi}_e$  by Harish-Chandra's Plancherel formula. We write 
 \[
 \begin{aligned}
& \widehat{\Phi}_{e}(\widehat{f}_0 , \dots,  \widehat{f}_n) \\
 =&(-1)^n\sum_{\tau \in S_n}\mathrm{sgn}(\tau) \int_{K N_\cc}\int_{G^{\times n} } H_{\tau(1)}\left(g_1 \dots g_n k\right)\dots H_{\tau(n)}\left(g_nk\right) \cdot f'  \cdot f_1(g_1) \dots f_n(g_n),\\
\end{aligned} 
\]
where the function $f'$ is defined the following formula
\[
\begin{aligned}
f'= &\int_{\eta^{M_\cc} \otimes \varphi \in (\widehat{M_\cc A_\cc} )_\text{temp}}   \int_{M_\cc A_\cc } e^{\langle \log \varphi+\rho,  \log a\rangle} \\
&\Theta^{M_\cc}(\eta)(m) \cdot  f_0(kmank^{-1} (g_1 g_2 \dots g_n)^{-1}) \cdot m(\eta^{M_\cc}).
\end{aligned}
\]
If we put
\[
c= e^{\langle \rho,  \log a \rangle}\cdot f_0\left(kmank^{-1} (g_1 g_2 \dots g_n)^{-1}\right),
\]
then
\[
f' = \int_{\eta^{M_\cc}\otimes \varphi \in (\widehat{M_\cc A_\cc} )_\text{temp}}   \left( \Theta^{M_\cc}(\eta) \otimes \varphi \right)(c) \cdot  m(\eta^{M_\cc}).
\]
By Theorem \ref{Planchere formula},
\[
f' = c\big|_{m = e, a = e} = f_0\left(knk^{-1} (g_1 g_2 \dots g_n)^{-1}\right).
\]
This completes the proof.

\subsection{Proof of Theorem \ref{main thm-3}}\label{subsec:proof-x}
Our proof  strategy  for Theorem \ref{main thm-3} is similar to Theorem \ref{main thm-2}. We split its proof into 3 steps as before. \\

\noindent $\mathbf{Step \ 1} \colon$
Let $\Lambda^*_{K \cap M_\cc}$ be the intersection of $\Lambda^*_T$ and the positive Weyl chamber for the group $M_\cc \cap K$. We denote by $\rho_c^{M_\cc \cap K}$ the half sum of positive roots in $\mathcal{R}(\kt, \km_\cc \cap \kk)$. For any $\lambda \in \Lambda^*_{K \cap M_\cc} $, we can find an element $w \in W_K /W_{K \cap M_\cc}$ such that $w \cdot \lambda \in \Lambda^*_{K}$. Moreover, for any $w \in W_K /W_{K \cap M_\cc}$, the induced representation
\[
\Ind_{P_\cc}^G(\sigma^{M_\cc}(\lambda)\otimes \varphi) \cong \Ind_{P_\cc}^G(\sigma^{M_\cc}(w \cdot \lambda)\otimes \varphi).
\]
The regular part 
\begin{equation}
\label{regular equ}
\begin{aligned}
& \sum_{\text{regular} \ \lambda \in \Lambda^*_{K} + \rho_c} \left[\widehat{\Phi}_{t}(\widehat{f}_0, \dots, \widehat{f}_n)\right]_{\lambda}\\
=& \sum_{\text{regular} \ \lambda \in \Lambda^*_{K}+ \rho_c} \left(\sum_{w \in W_{K}} (-1)^w \cdot e^{w \cdot \lambda}(t) \right) \cdot \int_{\varphi \in \widehat{A}_\cc}      T_{\Ind_{P_\cc}^G(\sigma^{M_\cc}(\lambda)\otimes \varphi)}(\widehat{f}_0 , \dots,   \widehat{f}_n)\cdot d\varphi\\
=& \sum_{\text{regular} \ \lambda \in \Lambda^*_{K\cap M_\cc}+ \rho_c} \left(\sum_{w \in W_{K \cap M_\cc}} (-1)^w \cdot e^{w \cdot \lambda}(t) \right) \cdot  \int_{\varphi \in \widehat{A}_\cc}      T_{\Ind_{P_\cc}^G(\sigma^{M_\cc}(\lambda)\otimes \varphi)}(\widehat{f}_0 , \dots,  \widehat{f}_n)\cdot  d\varphi.
\end{aligned}
\end{equation}
Remembering that the above is anti-invariant under the $W_K$-action, we can replace $\rho_c$ by $\rho_c^{M_\cc \cap K}$. That is, (\ref{regular equ}) equals to 
\[
\sum_{\text{regular} \ \lambda \in \Lambda^*_{K\cap M_\cc}+ \rho_c^{M_\cc \cap K}} \left(\sum_{w \in W_{K \cap M_\cc}} (-1)^w \cdot e^{w \cdot \lambda}(t) \right) \cdot  \int_{\varphi \in \widehat{A}_\cc}      T_{\Ind_{P_\cc}^G(\sigma^{M_\cc}(\lambda)\otimes \varphi)}(\widehat{f}_0, \dots,  \widehat{f}_n)\cdot  d\varphi.
\]
Similarly, the singular part 
\[
\begin{aligned}
& \sum_{\text{singular} \ \lambda \in \Lambda^*_{K}+ \rho_c} \left[\widehat{\Phi}_{t}(\widehat{f}_0, \dots, \widehat{f}_n)\right]_{\lambda}= \sum_{\text{singular} \ \lambda \in \Lambda^*_{K \cap M_\cc}+ \rho_c^{M_\cc \cap K}}  \left(\sum_{w \in W_{K \cap M_\cc}} (-1)^w \cdot e^{w \cdot \lambda}(t) \right)\\
& \times \left(\sum_{i=1}^{n(\lambda)} 
\frac{\epsilon(i)}{n(\lambda)} \cdot \int_{\varphi \in \widehat{A}_\cc}     T_{\Ind_{P_\cc}^G(\sigma^{M_\cc}_i(\lambda)\otimes \varphi)}(\widehat{f}_0 , \dots,   \widehat{f}_n) \right).
\end{aligned}
\]
At last, the higher part
\[
\begin{aligned}
&\left[\widehat{\Phi}_{t}(\widehat{f}_0, \dots, \widehat{f}_n)\right]_{\mathrm{high}}\\
=&\int_{\pi \in \widehat{G}^\text{high}_\mathrm{temp}} T_{\pi}(\widehat{f}_0 , \dots,  \widehat{f}_n)\big( \sum_{\eta^{M_\cc} \otimes \varphi \in \mathcal{A}(\pi)} \kappa^{M_\cc} (\eta^{M_\cc}, t)
\big) \cdot d\varphi\\
=& \int_{\eta^{M_\cc} \otimes \varphi \in \widehat{M}^\text{high}_\mathrm{temp} \times \widehat{A}_\cc} T_{\Ind_{P_\cc}^G(\eta^{M_\cc} \otimes \varphi)}(\widehat{f}_0, \dots,  \widehat{f}_n) \cdot \kappa^{M_\cc} (\eta^{M_\cc}, t). 
\end{aligned}
\]

\noindent $\mathbf{Step \ 2} \colon$ We apply  Proposition \ref{key prop} and obtain the following.
\begin{itemize}
\item	
regular part:
\[
\begin{aligned}
& \sum_{\text{regular} \ \lambda \in \Lambda^*_{K} + \rho_c} \left[\widehat{\Phi}_{t}(\widehat{f}_0, \dots, \widehat{f}_n)\right]_{\lambda}\\
 =&(-1)^n\sum_{\tau \in S_n}\mathrm{sgn}(\tau) \sum_{\text{regular} \ \lambda \in \Lambda^*_{K \cap M_\cc} + \rho_c^{M_\cc \cap K}} \left(\sum_{w \in W_{K \cap M_\cc}} (-1)^w \cdot e^{w \cdot \lambda}(t) \right)\\
 &\int_{\varphi \in \widehat{A}_\cc} \int_{KM_\cc A_\cc N_\cc}\int_{G^{\times n} } H_{\tau(1)}\big(g_1 \dots g_n k\big)\dots H_{\tau(n)}\big(g_nk\big)\\
 & e^{\langle 
 \log \varphi+\rho,  \log a\rangle} \cdot \Theta^{M_\cc}\big(\lambda\big)(m) \cdot f_0\big(kmank^{-1} (g_1 g_2 \dots g_n)^{-1}\big) f_1(g_1) \dots f_n(g_n).
 \end{aligned}
 \]
 \item	
singular part:
\[
\begin{aligned}
& \sum_{\text{singular} \ \lambda \in \Lambda^*_{K} + \rho_c} \left[\widehat{\Phi}_{t}(\widehat{f}_0, \dots, \widehat{f}_n)\right]_{\lambda}\\
 =&(-1)^n\sum_{\tau \in S_n}\mathrm{sgn}(\tau) \sum_{\text{singular} \ \lambda \in \Lambda^*_{K \cap M_\cc} + \rho_c^{M_\cc \cap K}} \sum_{i=1}^{n(\lambda)} \left(\frac{\epsilon(i)}{n(\lambda)}\sum_{w \in W_{K \cap M_\cc}} (-1)^w \cdot e^{w \cdot \lambda}(t) \right)\\
 &\int_{\varphi \in \widehat{A}_\cc} \int_{KM_\cc A_\cc N_\cc}\int_{G^{\times n} } H_{\tau(1)}\left(g_1 \dots g_n k\big)\dots H_{\tau(n)}\big(g_nk\right)\\
 & e^{\langle \log \varphi+\rho, \log a\rangle} \cdot \Theta^{M_\cc}_i(\lambda)(m) \cdot f_0\left(kmank^{-1} (g_1 g_2 \dots g_n)^{-1}\right) f_1(g_1) \dots f_n(g_n).
 \end{aligned}
 \]
 \item
 higher part:
\[
\begin{aligned}
&\left[\widehat{\Phi}_{t}(\widehat{f}_0, \dots, \widehat{f}_n)\right]_{\mathrm{high}}\\
=&(-1)^n\int_{\eta^{M_\cc} \otimes \varphi \in \widehat{M}^\text{high}_\mathrm{temp} \times \widehat{A}_\cc} \int_{KM_\cc A_\cc N_\cc}\int_{G^{\times n} } H_{\tau(1)}\big(g_1 \dots g_n k\big)\dots H_{\tau(n)}\big(g_nk\big)\\
 & e^{\langle \log \varphi+\rho, \log a\rangle} \cdot \Theta^{M_\cc}\big(\eta\big)(m) \cdot f_0\big(kmank^{-1} (g_1 g_2 \dots g_n)^{-1}\big)\\
 & f_1(g_1) \dots f_n(g_n) \cdot \kappa^{M_\cc}(\eta^{M_\cc}, t).  
\end{aligned}
\] 
 \end{itemize}

\noindent $\mathbf{Step \ 3} \colon$ Combining all the above computation together, we have 
\begin{equation}
\label{step 3}	
\begin{aligned}
	&\widehat{\Phi}_{t}(\widehat{f}_0 , \dots,  \widehat{f}_n)\\
	=& \sum_{\text{regular} \ \lambda \in \Lambda^*_{K} + \rho_c} \left[\widehat{\Phi}_{t}(\widehat{f}_0, \dots, \widehat{f}_n)\right]_{\lambda} + \sum_{\text{singular} \ \lambda \in \Lambda^*_{K} + \rho_c} \left[\widehat{\Phi}_{t}(\widehat{f}_0, \dots, \widehat{f}_n)\right]_{\lambda}
+ \left[\widehat{\Phi}_{t}(\widehat{f}_0, \dots, \widehat{f}_n)\right]_{\mathrm{high}}\\
=&(-1)^n \int_{K N_\cc}\int_{G^{\times n} }  f' \cdot  \left( \sum_{\tau \in S_n}\mathrm{sgn}(\tau) \cdot H_{\tau(1)}\big(g_1 \dots g_n k\big)\dots H_{\tau(n)}\big(g_nk\big)\cdot    f_1(g_1) \dots f_n(g_n) \right),\\
\end{aligned}
\end{equation}
where
\[
\begin{aligned}
f' &= \sum_{\text{regular} \ \lambda \in \Lambda^*_{K \cap M_\cc} + \rho_c^{M_\cc \cap K}} \left(\sum_{w \in W_{K \cap M_\cc}} (-1)^w \cdot e^{w \cdot \lambda}(t) \right) \cdot \int_{\varphi \in \widehat{A}_\cc}  \left(\Theta^{M_\cc}(\lambda)\otimes \varphi\right)(c)\\
+& \sum_{\text{singular} \ \lambda \in \Lambda^*_{K \cap M_\cc} + \rho_c^{M_\cc \cap K}} \left(\frac{\sum_{w \in W_{K \cap M_\cc}} (-1)^w \cdot e^{w \cdot \lambda}(t) }{n(\lambda)}\right) \cdot \sum_{i=1}^{n(\lambda)} \epsilon(i) \cdot \int_{\varphi \in \widehat{A}_\cc} \left(\Theta^{M_\cc}_i(\lambda)\otimes \varphi\right)(c) \\
+& \int_{\eta^{M_\cc}\otimes \varphi \in \widehat{M}^\text{high}_\mathrm{temp} \otimes \widehat{A}_\cc} \left(\Theta^{M_\cc}(\eta)\otimes \varphi \right)(c) \cdot \kappa^{M_\cc}(\eta^{M_\cc}, t),
\end{aligned}
\]
where
\[
c= e^{\langle \rho, \log a \rangle }\cdot f_0\big(kmank^{-1} (g_1 g_2 \dots g_n)^{-1}\big).
\]
We then apply Theorem \ref{orbital formula} to  the function $c$. Hence we obtain that 
\begin{equation}
\label{f'}
f' = F^T_c(t) = \Delta^{M_\cc}_T(t) \cdot \int_{h \in M_\cc/T_\cc} f_0\left(k h t h^{-1} n k^{-1} (g_1 g_2 \dots g_n)^{-1} \right). 
\end{equation}
Hence, by (\ref{step 3}) and (\ref{f'}), we conclude that 
\[
\begin{aligned}
	&\widehat{\Phi}_{t}(\widehat{f}_0 , \dots,   \widehat{f}_n)\\
=&  \Delta^{M_\cc}_T(t) \cdot  (-1)^n\sum_{\tau \in S_n}\mathrm{sgn}(\tau) \int_{h \in M_\cc/T_\cc} \int_{K N_\cc}\int_{G^{\times n} } H_{\tau(1)}\big(g_1 \dots g_n k\big)\dots H_{\tau(n)}\big(g_nk\big)\\
& f_0\big(k hth^{-1} n k^{-1} (g_1 g_2 \dots g_n)^{-1} \big)  f_1(g_1) \dots f_n(g_n)\\
=&(-1)^n \Delta^{M_\cc}_T(t) \cdot \Phi_t(f_0, \dots, f_n). 
\end{aligned}
\]
This completes the proof.

\section{Higher Index Pairing}\label{sec:higher-index}
In this section, we  study $K_0(C^*_r(G))$ by computing its pairing with $\Phi_t$ , $t\in T^{\text{reg}}$ and $\Phi_e$. And we construct a group isomorphism $\cF \colon  K_0(C^*_r(G))\to \text{Rep}(K)$. 

\subsection{Generators of  $K(C^*_r(G))$}
\label{sec: sing generator}

In Theorem \ref{thm:k-theory-G}, we explain that the $K$-theory group of $C_r^*(G)$ is a free abelian group generated by the following components, i.e.
\begin{equation}
\begin{aligned}
K_0(C^*_r(G))\cong& \bigoplus_{[P, \sigma]^{\ess}} K_0 \left(\mathcal{K}\big(C^*_r(G)_{[P, \sigma]}\right)\\
\cong & \bigoplus_{ \lambda \in \Lambda^*_{K} +\rho_c} \mathbb{Z}.
\end{aligned}
\end{equation}

Let $[P, \sigma] \in \cP(G)$ be an essential class corresponds to $\lambda \in \Lambda^*_{K} + \rho_c$. In this subsection, we construct a generator of $K_0(C^*_r(G))$ associated to $\lambda$. 
We decompose $\widehat{A}_P = \widehat{A}_S \times \widehat{A}_\cc$ and denote $r = \dim \widehat{A}_S$ and $ \widehat{A}_\cc = n$. By replacing $G$ with $G\times \mathbb{R}$, we may assume that $n$ is even. 

Let $V$ be an $r$-dimensional complex vector space and $W$ an $n$-dimensional Euclidean space. Take
\[
z = (x_1, \cdots, x_r, y_1, \dots y_{n}), \quad x_i \in \C, y_j \in \R
\] 
to be coordinates on $V\oplus W$. Assume that the finite group $(\mathbb{Z}_2)^{r}$ acts on $W$ by simple reflections. In terms of coordinates, 
\[
(x_1, \cdots, x_r, y_1, \dots, y_{n})\mapsto (\pm x_1, \cdots, \pm x_r,  y_1, \dots, y_{n}). 
\] 

Let us consider the the Clifford algebra 
\[
\mathrm{Clifford}(V) \otimes \mathrm{Clifford} (W)
\]
together with the spinor module $S = S_{V} \otimes S_{W}$. Here the spinor modules are equipped with a $\Z_2$-grading:
\[
S^+ = S_{V} ^+\otimes S_{W}^+ \oplus S_{V} ^-\otimes S_{W}^-, \quad S^- = S_{V} ^+\otimes S_{W}^- \oplus S_{V} ^-\otimes S_{W}^+. 
\]

Let $\mathcal{S}(V), \mathcal{S}(W)$, and $\mathcal{S}(V\oplus W)$ be the algebra of Schwarz functions on $V$, $W$, and $V\oplus W$.  The Clifford action $c(z) \colon S^\pm \to S^\mp$
\[
c(z) \in  \mathrm{End}(S_V) \otimes \mathrm{End}(S_W). 
\]
 Let $e_1, \dots e_{2^{r-1}}$ be a basis for $S^+_V$, $e_{2^{r-1}+1}, \dots e_{2^{r}}$ be a basis for $S^-_V$, and $f_1, \dots f_{2^{\frac{n}{2}}}$ a basis for $S_W$. We write 
\[
c_{i, j, k, l}(z) = \langle c(z) e_i \otimes f_l, e_j \otimes f_k \rangle, \quad 1 \leq i, j \leq 2^r, 1 \leq k, l \leq 2^{\frac{n}{2}},
\]
and define 
\[
T:=\left(
			\begin{array}{cc}
				e^{-|z|^2} \cdot \mathrm{id}_{S^+}& e^{-\frac{|z|^2}{2}}(1-e^{-|z|^2})\cdot \frac{c(z)}{|z|^2}\\
				e^{-\frac{|z|^2}{2}}c(z) & (1-e^{-|z|^2}) \cdot \mathrm{id}_{S^-}
				\end{array}
				\right) - \left(\begin{array}{cc}
				0& 0\\
				0 & \mathrm{id}_{S^-}
				\end{array}
				\right) ,  
\]
which is a $2^{r+\frac{n}{2}} \times 2^{r+\frac{n}{2}}$ matrix:  
\[
 \big(t_{i,j, k, l}\big), \quad 1 \leq i, j \leq 2^r, 1 \leq k, l \leq 2^{\frac{n}{2}},
\]
with $t_{i, j, k, l} \in \mathcal{S}(V \oplus W)$. 

\begin{definition}
On the $n$-dimensional Euclidean space $W$, we can define
\[
B^{n}=\left(
			\begin{array}{cc}
				e^{-|y|^2} \cdot \mathrm{id}_{S_{W} ^+}& e^{-\frac{|y|^2}{2}}(1-e^{-|y|^2})\cdot \frac{c(y)}{|y|^2}\\
				e^{-\frac{|y|^2}{2}}c(y) & (1-e^{-|y|^2}) \cdot \mathrm{id}_{S_{W} ^-}
				\end{array}
				\right) - \left(\begin{array}{cc}
				0& 0\\
				0 & \mathrm{id}_{S_{W}^-}
				\end{array}
				\right),  
\]
which is a $2^{\frac{n}{2}}\times 2^{\frac{n}{2}}$ matrix:  
\[
 \big(b_{k, l}\big), \quad 1 \leq k, l \leq 2^{\frac{n}{2}},
\]
with $b_{k, l} \in \mathcal{S}(W)$. Then $B^{n}$ is the \emph{Bott generator} in $K_0(C_0(W)) \cong \Z$.
	
\end{definition}

\begin{lemma}\label{x= 0 lem}
If we restrict to $W \subset V \oplus W$, that is $x = 0$, then we have that 
\[
T\big|_{x=0} = \left(\begin{array}{cc}
				\mathrm{id}_{S_{V}^+} & 0\\
				0 & -\mathrm{id}_{S_{V}^-}
				\end{array}
				\right) \otimes B^{n}. 
\]
\begin{proof}
By definition, we have that 
\begin{itemize}
\item
$t_{i, j, k, l} = e^{-z^2}$ when $e_i, e_j, f_k, f_l \in S^+$; 
\item
$t_{i, j, k, l} = -e^{-z^2}$ when $e_i, e_j, f_k, f_l \in S^-$;
\item
$t_{i, j, k, l} = e^{-\frac{|z|^2}{2}}(1-e^{-|z|^2})\cdot \frac{c_{i, j, k, l}(z)}{|z|^2}$ when $e_i, f_k, \in S^+$ and $e_j, f_l, \in S^-$; 
\item
$t_{i, j, k, l} = e^{-|z|^2} \cdot c_{i, j, k, l}(z)$ when $e_i, f_k, \in S^-$ and $e_j, f_l, \in S^+$.
\end{itemize}
Moreover, the Clifford action 
\[
c(z) = c(x) \otimes 1 + 1 \otimes c(y) \in  \mathrm{End}(S_V) \otimes \mathrm{End}(S_W)
\]
for $z = (x, y) \in V \oplus W$. Thus, $c_{i,j, k, l}(z)\big|_{x = 0} = c_{k, l}(y)$. This completes the proof. 
\end{proof}
\end{lemma}

On the other hand, the induced representation decomposes
\[
 \Ind_{P}^G(\sigma \otimes \varphi \otimes 1) = \bigoplus_{i=1}^{2^r} \Ind_{P}^G \big( \delta_i \otimes \varphi \big).
\]
By Equation (\ref{char of limit}), the characters of the limit of discrete series representations of $\delta_i$ are all the same up to a sign after restricting to a compact Cartan subgroup of $M_P$. We can organize the numbering so that 
\[
\delta_i, \quad i = 1, \dots 2^{r-1}
\]
have the same character after restriction and 
\[
\delta_i, \quad i = 2^{r-1}, \dots 2^{r}
\]
share the same character. In particular, $\delta_i$ with $ 1 \leq i \leq 2^{r-1}$ and $\delta_j$ with $2^{r-1}+1\leq j \leq 2^r$ have the opposite characters after restriction.

We fix $2^r$ unit $K$-finite vectors $v_i \in \Ind_{P}^G \big( \delta_i \big)$, and define 	
\begin{equation}
\label{generator S}
S_\lambda := \big(t_{i,j, k, l}\cdot v_i \otimes v_j^*\big).
\end{equation}
The matrix
\[
S_\lambda \in \Big[\calS\big(\widehat{A}_P, \calL(\ind_P^G\sigma)\big) \Big]^{W_\sigma},
\]
and it is an idempotent. Recall that
\[ 
\mathcal{K}\big(\Ind_P^G \sigma \big)^{W_\sigma} \sim \big(C_0(\R)\rtimes \Z_2\big)^r \otimes C_0(\R^{n}).
\]

\begin{definition} \label{defn:Qn-generator}
We define 
\[
Q_\lambda \in M_{2^{r+ n}}(\mathcal{S}(G))
\]
to be the wave packet associated to $S_\lambda$. Then $[Q_\lambda]$ is the generator in $ K_0 \left(C^*_r(G)_{[P, \sigma]}\right)$ for $[P, \sigma]$. 
\end{definition}

\subsection{The main results}
Let $G$ be a connected,  linear, real,  reductive Lie group with maximal compact subgroup $K$. We denote by $T$ the maximal torus of $K$, and $P_\cc$ the maximal parabolic subgroup of $G$. It follows from Appendix \ref{sec: description} that for any $\lambda \in \Lambda^*_{K} + \rho_c$, there is generator 
\[
[Q_\lambda] \in K\big(C^*_r(G)\big).
\] 
In Section \ref{sec:cyclic-cocycle}, we defined a family of cyclic cocycles 
\[
\Phi_{e}, \hspace{5mm}  \Phi_{t} \in HC\big(\mathcal{S}(G)\big) 
\]
for all  $t\in T^{\text{reg}}$.  

\begin{theorem}\label{thm:pairing}
The index pairing between cyclic cohomology and $K$-theory
\[
HP^{\text{even}}\big(\mathcal{S}(G)\big)  \otimes K_0\big(\mathcal{S}(G)\big) \to \C
\] 
is given by the following formulas:
\begin{itemize}
	\item 
 We have 
\[
\langle \Phi_{e}, [Q_\lambda] \rangle =  \frac{(-1)^{\dim(A_\cc)}}{|W_{M_\cc \cap K}|} \cdot \sum_{w \in W_K} m\left(\sigma^{M_\cc}(w \cdot \lambda)\right),
\] 
where $\sigma^{M_\cc}(w \cdotp \lambda)$ is the discrete series representation with Harish-Chandra parameter $w \cdot \lambda$, and $m\left(\sigma^{M_\cc}(w \cdot \lambda)\right)$ is its Plancherel measure;
\item
For any $t \in T^\text{reg}$, we have that 
\begin{equation}\label{eq:orbit-integral}
\langle \Phi_{t}, [Q_\lambda] \rangle =  (-1)^{\dim(A_\cc)}\frac{\sum_{w \in W_K} (-1)^w e^{w \cdot \lambda}(t)}{\Delta^{M_\cc}_T(t)}.
\end{equation}
\end{itemize}
\end{theorem}

The proof of Theorem \ref{thm:pairing} is presented in Sections \ref{subsec:regular} and \ref{subsec:singular}.

\begin{corollary}\label{cor:character}
The index paring of $[Q_\lambda]$ and \emph{normalized higher orbital integral} equals to the character of the representation $\Ind_{P_\cc}^G(\sigma^{M_\cc}(\lambda) \otimes \varphi)$ at $\varphi = 1$. That is, 
\[
\left\langle \frac{\Delta^{M_\cc}_T}{\Delta^{G}_T}\cdot \Phi_{t}, [Q_\lambda] \right\rangle = (-1)^{\dim(A_\cc)}\Theta(P_\cc, \sigma^{M_\cc}(\lambda), 1)(t). 
\]	
\begin{proof}
It follows from applying the character formula, Corollary \ref{coro character}, to the right side of Equation (\ref{eq:orbit-integral}). 
\end{proof}
\end{corollary}

\begin{remark}\label{rmk:geometry}
If the group $G$ is of equal rank, then the normalization factor is trivial. And the above corollary says that the orbital integral equals to the character of (limit of ) discrete series representations. This result in the equal rank case is also obtained by Hochs-Wang in \cite{MR3990784} using fixed point theorem and the Connes-Kasparov isomorphism.  In contrast to the Hochs-Wang approach, our proof is based on representation theory and does not use any geometry of the homogenous space $G/K$ or the Connes-Kasparov theory.  
\end{remark}

We notice that though the cocycles $\Phi_t$ introduced in Definition \ref{defn:Phi} are only defined for regular elements in $T$, Theorem \ref{thm:pairing} suggests that the pairing $\Delta^{M_\cc}_T(t)  \langle \Phi_t, [Q_\lambda]\rangle$ is a well defined smooth function on $T$.  This inspires us to introduce the following map. \begin{definition} \label{defn:isom}Define a map $\cF^T\colon K_0(C^*_r(G))\to C^\infty(T)$ by 
\[
 \cF^T([Q_\lambda])(t) \colon =  (-1)^{\dim (A_\cc)} \cdot \frac{ \Delta^{M_\cc}_T}{\Delta^{K}_T}\cdot  \langle \Phi_t, [Q_\lambda]\rangle.
\] 
\end{definition}
Let $\Rep(K)$ be the character ring of the maximal compact subgroup $K$. By the Weyl character formula, for any irreducible $K$-representation $V_\lambda$ with highest weight $\lambda$, its character is given by 
\[
\Theta_\lambda(t) = \frac{\sum_{w \in W_K} (-1)^w e^{w \cdot (\lambda+\rho_c)}(t)}{\Delta^{K}_T(t)}.
\]
\begin{corollary}\label{cor:isom}
The map $\cF^T: K_0(C^*_r(G))\to \Rep(K)$ is an isomorphism of abelian groups. 
\end{corollary}
In \cite{Clare-Higson-Song-Tang}, we will use the above property of $\cF^T$ to show that $\cF^T$  is actually the inverse of the Connes-Kasparov Dirac index map, $\ind: \Rep(K)\to K(C^*_r(G))$.  

\begin{remark}\label{rmk:cyclic} The cyclic homology of the algebra $\mathcal{S}(G)$ is studied by Wassermann \cite{MR996447}. Wassermann's result together with Corollary \ref{cor:isom} show that higher orbit integrals $\Phi_t, t\in H^{\text{reg}}$,  generate $HP^{\bullet}(\mathcal{S}(G))$. Actually, Definition \ref{defn:Phi} can be generalized to construct a larger class of Hochschild cocycles for $\mathcal{S}(G)$ for every parabolic subgroup (not necessarily maximal) with different versions of Theorem I and II. We plan to investigate this general construction and its connection to the Plancherel theory in the near future. 
\end{remark}

\subsection{Regular case}\label{subsec:regular}
Suppose that $\lambda \in \Lambda^*_{K} + \rho_c$ is regular and $\sigma^{M_\cc}(\lambda)$ is the discrete series representation of $M_\cc$ with Harish-Chandra parameter $\lambda$. We consider the generator  $[Q_\lambda]$, the wave packet associated to the matrix $S_\lambda$ introduced in   (\ref{generator S}),  corresponding to 
\[
 \Ind_{P_\cc}^G(\sigma^{M_\cc}(\lambda) \otimes \varphi), \hspace{5mm} \varphi \in \widehat{A}_\cc.
\]
According to Theorem \ref{main thm-2}, we have that 
\[
\begin{aligned}
(-1)^n\langle \Phi_{e}, Q_\lambda \rangle =&\int_{\pi \in \widehat{G}_\mathrm{temp}}     T_{\pi}\left( \tr \left(\underbrace{S_\lambda \otimes \dots \otimes S_\lambda}_{n+1} \right) \right)\cdot  \mu(\pi )  \\
=&  \int_{\widehat{A}_\cc}     T_{\Ind_{P_\cc}^G(\sigma^{M_\cc}(\lambda) \otimes \varphi) } \cdot \left( \tr \left(\underbrace{S_\lambda \otimes \dots \otimes S_\lambda}_{n+1} \right) \right) \cdot \mu\left(\Ind_{P}^G(\sigma^{M_\cc}(\lambda) \otimes \varphi)\right).\\ 
\end{aligned}
\]
By definition, 
\[
\begin{aligned}
\mu\left(\Ind_{P_\cc}^G(\sigma^{M_\cc}(\lambda) \otimes \varphi)\right) =& \mu\left(\Ind_{P_\cc}^G(\sigma^{M_\cc}(\lambda) \otimes 1\right) \\
=& \sum_{w \in W_K/ W_{K \cap M_\cc}} m\left(\sigma^{M_\cc}(w \cdot \lambda)\right)\\
=&\frac{1}{|W_{K\cap M_\cc}|} \cdot \sum_{w \in W_K} m\left(\sigma^{M_\cc}(w \cdot \lambda)\right).
\end{aligned}
\]
Moreover, in the case of regular $\lambda$, $S_\lambda = [B^n \cdot (v \otimes v^*)]$, where $B^n$ is the Bott generator for $K_0(\mathcal{S}(\widehat{A}_\cc))$ and $v$ is a unit $K$-finite vector in $\Ind_{P_\cc}^G(\sigma^{M_\cc}(\lambda))$. By (\ref{Tpi equ}), 
\[
\begin{aligned}
&\int_{\widehat{A}_\cc}     T_{\Ind_{P_\cc}^G(\sigma(\lambda) \otimes \varphi) }\left( \tr \big(\underbrace{S_\lambda \otimes \dots \otimes S_\lambda}_{n+1} \big) \right)\\
=& \langle B^n, b_{n} \rangle = 1
\end{aligned} 
\]
where $[b_n] \in HC^n(\mathcal{S}(\R^n))$ is the cyclic cocycle on $\mathcal{S}(\R^n)$ of degree $n$, c.f. Example \ref{ex:Rn-cocycle}. 
We conclude that 
\[
\langle \Phi_{e}, [Q_\lambda] \rangle = \frac{(-1)^n}{|W_{K\cap M_\cc}|} \cdot \sum_{w \in W_K} m\left(\sigma^{M_\cc}(w \cdot \lambda)\right).
\]

For the orbital integral, only the regular part will contribute. The computation is similar as above, and we conclude that 
\[
\begin{aligned}
\langle \Phi_{t}, Q_\lambda \rangle =&\left( (-1)^n\sum_{w \in W_K}  (-1)^w e^{w \cdot \lambda}(t) \right) \cdot  \int_{\widehat{A}_\cc}     T_{\Ind_{P_\cc}^G(\sigma^{M_\cc}(\lambda) \otimes \varphi) }\left( \tr \left(\underbrace{S_\lambda \otimes \dots \otimes S_\lambda}_{n+1} \right) \right)\\
=&(-1)^n\sum_{w \in W_K}  (-1)^w e^{w \cdot \lambda}(t). 
\end{aligned}
\]

\subsection{Singular case}\label{subsec:singular}

Suppose now that $\lambda \in \Lambda^*_{K} + \rho_c$ is singular. We decompose 
\[
\widehat{A}_{P} = \widehat{A}_\cc \times \widehat{A}_S, \hspace{5mm} \varphi = (\varphi_1, \varphi_2). 
\]
We denote $r = \dim(A_S)$ and $n = \dim(A_\cc)$ as before. In this case, we have that 
\[
 \Ind_{P}^G(\sigma^M \otimes \varphi_1 \otimes 1) =   \bigoplus_{i=1}^{2^r} \Ind_{P_\cc}^G \big( \sigma^{M_\cc}_i \otimes \varphi_1 \big),
\]
where $\sigma^M$ is a discrete series representation of $M$ and  $ \sigma^{M_\cc}_i $, $i=1,..., 2^r$, are limit of discrete series representations of $M_\cc$ with Harish-Chandra parameter $\lambda$. 

Recall that the generator $Q_\lambda$ is the wave packet associated to $S_\lambda$. The index paring equals to 
\[
(-1)^n\langle \Phi_{e}, Q_\lambda \rangle = \int_{\widehat{A}_P}      T_{\Ind_{P}^G(\sigma^M\otimes \varphi)}\left( \tr \left(\underbrace{S_\lambda \otimes \dots \otimes S_\lambda}_{n+1} \right) \right)\cdot \mu\left( \Ind_{P}^G(\sigma^M \otimes \varphi) \right). 
\]
By the definition of $\mu$, 
\[
\mu\left( \Ind_{P}^G(\sigma^M \otimes \varphi) \right) = \sum_{\eta^{M_\cc} \otimes \varphi_1 \in \mathcal{A}\left(( \Ind_{P}^G(\sigma^M  \otimes \varphi) \right)} m(\eta^{M_\cc}).
\]
Thus, the function $\mu\big( \Ind_{P}^G(\sigma^M \otimes \varphi) \big)$ is constant with respect to $\varphi_1 \in\widehat{A}_\cc$. It follows from (\ref{Tpi equ}) that 

\[
\begin{aligned}
&(-1)^n\langle \Phi_{e}, Q_\lambda \rangle \\
=& \int_{\widehat{A}_S}  \left(\mu\left( \Ind_{P}^G(\sigma^M \otimes \varphi) \right)     \cdot   \int_{\widehat{A}_\cc}  T_{\Ind_{P}^G(\sigma^M, \varphi)}\left( \tr (\underbrace{S_\lambda \otimes \dots \otimes S_\lambda}_{n+1} ) \right)\right) \\
=& \sum_{\tau \in S_n} \sum_{j_0 = i_1, \dots j_{n-1} = i_n, j_n = i_0} \\
&\qquad \qquad \int_{\widehat{A}_S} \left( \mu\big( \Ind_{P}^G(\sigma^M \otimes \varphi)  \int_{\widehat{A}_\cc} (-1)^\tau \cdot s_{i_0, j_0}(\varphi) \frac{\partial s_{i_1, j_1}(\varphi)}{\partial_{\tau(1)}}  
\dots \frac{\partial s_{i_n, j_n}(\varphi)}{\partial_{\tau(n)}} \right), 
\end{aligned}
\]
where $s_{i, j} \in \calS(\widehat{A}_\cc \times \widehat{A}_S)$ with $1 \leq i, j \leq 2^{r + \frac{n}{2}}$ is the coefficient of the $(i,j)$-th entry in the matrix $S_\lambda$. We notice that the dimension of $\widehat{A}_P$ is $n+r$. It follows from the Connes-Hochschild-Kostant-Rosenberg theorem that the periodic cyclic cohomology of the algebra $\mathcal{S}(\widehat{A}_P)$ is spanned by a  cyclic cocycle of degree  $n+r$. Accordingly,  we conclude that 
\[
\langle [\Phi_{e}], Q_\lambda \rangle  = 0
\]
because it equals to the pairing of the Bott element $B^{n+r} \in K(\widehat{A}_P)$ and a cyclic cocycle in $HC(\widehat{A}_P)$ with degree only $n<n+r$.

Next we turn to the index pairing of orbital integrals. In this singular case, it is clear that the regular part of higher orbital integrals will not contribute. For the higher part, 
\[
\begin{aligned}
&(-1)^n\langle [\Phi_{t}]_\text{high}, Q_\lambda \rangle \\
=& \int_{\varphi \in \widehat{A}_P}      T_{\Ind_{P}^G(\sigma^M \otimes \varphi)}\left( \tr \left(\underbrace{S_\lambda \otimes \dots \otimes S_\lambda}_{n+1} \right) \right) \cdot \left( \sum_{\eta^{M_\cc} \otimes \varphi_1 \in \mathcal{A}(\Ind_{P}^G(\sigma^M \otimes \varphi))} \kappa^{M_\cc} (\eta^{M_\cc}, t)
\right).
\end{aligned}
\]
Note that the function 
\[
 \sum_{\eta^{M_\cc} \otimes \varphi_1 \in \mathcal{A}(\Ind_{P}^G(\sigma^M \otimes \varphi))} \kappa^{M_\cc} (\eta^{M_\cc}, t)
 \]
is constant on $\varphi_1 \in \widehat{A}_\cc$. By  (\ref{Tpi equ}), we can see that  
\[
\begin{aligned}
&\int_{\widehat{A}_S} \left( \sum_{\eta^{M_\cc} \otimes \varphi_1 \in \mathcal{A}\left(\Ind_{P}^G\left(\sigma^M \otimes \varphi \right)\right)} \kappa^{M_\cc} \left(\eta^{M_\cc}, t\right)  \cdot \int_{ \widehat{A}_\cc}     T_{\Ind_{P}^G\left(\sigma^M \otimes \varphi\right)}\tr \left(\underbrace{S_\lambda \otimes \dots \otimes S_\lambda}_{n+1}\right) \right)\\
=& \sum_{\tau \in S_n} \sum_{j_0 = i_1, \dots j_{n-1} = i_n, j_n = i_0} \int_{\widehat{A}_S} \Bigg( \sum_{\eta^{M_\cc} \otimes \varphi_1 \in \mathcal{A}(\Ind_{P}^G(\sigma^M \otimes \varphi))} \kappa^{M_\cc} (\eta^{M_\cc}, t)  \\
& \int_{\widehat{A}_\cc} (-1)^\tau \cdot s_{i_0, j_0}(\varphi) \frac{\partial s_{i_1, j_1}(\varphi)}{\partial_{\tau(1)}}  \dots \frac{\partial s_{i_n, j_n}(\varphi)}{\partial_{\tau(n)}} \Bigg). 
\end{aligned}
\]
 We conclude that 
\[
\langle [\Phi_{t}]_\text{high}, Q_\lambda \rangle  = 0
\]
because it equals to the paring of the Bott element $B^{n+r} \in K(\widehat{A}_P)$ and a cyclic cocycle on $\mathcal{S}(\widehat{A}_P)$ of degree $n$, which is trivial in $HP^{\text{even}}(\mathcal{S}(\widehat{A}_P))$.

For the singular part, by Schur's orthogonality, we have $\langle [\Phi_{t}]_{\lambda'}, Q_\lambda \rangle = 0$ unless $\lambda' = \lambda$. When $\lambda'=\lambda$, Theorem \ref{main thm-3} gives us the following computation, 
\[
\begin{aligned}
&(-1)^n\langle [\Phi_{t}]_{\lambda}, Q_\lambda \rangle\\
=&	\left( \frac{1}{2^r}\sum_{w \in W_K}  (-1)^w e^{w \cdot \lambda}(t) \right)  \cdot  \sum_{k=1}^{2^r}\int_{\varphi \in \widehat{A}_P}  \epsilon(k) \cdot   T_{\Ind_{P_\cc}^G(\sigma_k^{M_\cc} \otimes \varphi_1)}\left( \tr \left(\underbrace{S_\lambda \otimes \dots \otimes S_\lambda}_{n+1} \right) \right). 
\end{aligned}
\]
For each fixed $k$, it follows from Lemma \ref{x= 0 lem} and Lemma \ref{Tpi lem-2} that
\[
\begin{aligned}
&\int_{\varphi \in \widehat{A}_P}      T_{\Ind_{P_\cc}^G(\sigma_k^{M_\cc} \otimes \varphi_1)}\left( \tr \left(\underbrace{S_\lambda \otimes \dots \otimes S_\lambda}_{n+1} \right) \right)\\
=&\int_{\varphi \in \widehat{A}_P}      T_{\Ind_{P_\cc}^G(\sigma_k^{M_\cc} \otimes \varphi_1)}\left( \tr \left(\underbrace{S_\lambda \otimes \dots \otimes S_\lambda}_{n+1} \right)\Big|_{\varphi_2=0} \right)\\
 =&  \sum_{j_0 = i_1, \dots j_{n-1} = i_n, j_n = i_0 = k}  \sum_{\tau \in S_n}\int_{\varphi_1 \in \widehat{A}_\cc} (-1)^\tau \cdot s_{i_0, j_0}(\varphi_1) \frac{\partial s_{i_1, j_1}(\varphi_1)}{\partial_{\tau(1)}}  
\dots \frac{\partial s_{i_1, j_1}(\varphi_2)}{\partial_{\tau(1)}}\\
=& \begin{cases} 
      \langle B^n, b_n\rangle = 1 &  k = 1, \dots 2^{r-1} \\
      -\langle B^n, b_n\rangle = 1 & k = 2^{r-1}+ 1, \dots 2^{r}.  \\
   \end{cases}
\end{aligned}
\]
Combining all the above together and the fact that $\epsilon(k) = 1$ for $k = 1, \dots 2^{r-1}$ and $\epsilon(k) = -1$ for $k = 2^{r-1}+1, \dots 2^{r}$, we conclude that 
\[
 \langle [\Phi_{t}], Q_\lambda \rangle = \langle [\Phi_{t}]_{\lambda}, Q_\lambda \rangle = (-1)^n\sum_{w \in W_K}  (-1)^w e^{w \cdot \lambda}(t). 
\]

\vskip 2mm

\appendix

\noindent{\bf \Large Appendix}

\section{Integration of Schwartz functions}\label{sec:schwartz}

Let $\ka \subseteq \kp$ be the maximal abelian subalgebra of $\kp$ and $\kh = \kt \oplus \ka$ be the most non-compact Cartan subalgebra of $\kg$. Let $\mathfrak{u} = \kk \oplus i \kp$ and $U$ be the compact Lie group with Lie algebra $\mathfrak{u}$. Take  $v \in \ka^*$ an integral weight. Let $\tilde{v} \in \kt^* \oplus i \ka^*$ be an integral weight so that its restriction $\tilde{v}\big|_{i\ka^*} = i \cdot v$. Let $G^\C$ be the complexification of $G$. Suppose that $V$ be a finite-dimensional irreducible holomorphic representation of $G^\C$ with highest weight $\tilde{v}$. Introduce a Hermitian inner product $V$ so that $U$ acts on $V$ unitarily.

  We take $u_v$ to be a unit vector in the sum of the weight spaces for weights that restrict to $v$ on $\ka$. 
  
  \begin{lemma}
  For any $g \in G$, we have that 
  \[
  e^{\langle v, H(g)\rangle} = \|g \cdot u_v\|. 
  \]	
  \begin{proof}
  The proof is borrowed from \cite[Proposition 7.17]{MR1880691}. By the Iwasawa decomposition, we write $g = ka n$ with $a = \exp (X)$ and $X \in \ka$. Since $u_v$ is the highest vector for the action of $\ka$,  $\kn$ annihilates $u_v$. Thus, 
  \[
  \| g u_v\| = \|ka u_v\|= e^{\langle v, X \rangle }\|k u_v\|= e^{\langle v, X \rangle }. 
  \] 
  The last equation follows from the fact that $K \subseteq U$ acts on $V$ in a unitary way. On the other hand, we have that $H(g) = X$. This completes the proof. 
  \end{proof}
\end{lemma}

\begin{proposition}
\label{H est}
There exists a constant $C_v >0$ such that 
	\[
	\langle v, H(g) \rangle \leq C_v \cdot \| g\|,  
	\]
	where $\|g\|$ is the distance from $g \cdot K$ to $e \cdot K$ on $G/K$. 
\begin{proof}
Since $G = K\exp(\ka^+)K$, we write $g = k' \exp (X) k$ with $X \in \ka^+$. By definition, 
\[
\|g \| = \|X\|, \quad \text{and} \quad H(g) = H(ak). 
\]
By the above lemma, we have that 
\[
  e^{\langle v, H(ak)\rangle} = \|ak \cdot u_v\|. 
  \]	
  We decompose $k \cdot u_v$ into the weight spaces of $\ka$-action. That is 
  \[
  k \cdot u_v = \sum_{i=1}^n c_i \cdot u_i, 
  \]
  where $c_i \in \C, \|c_i\| \leq 1$ and $u_i$ is a unit vector in the weight spaces for weights that restricts to $\lambda_i \in \ka^*$. It follows that 
  \begin{equation}
  \label{cocycle equ}
  \begin{aligned}
  \|a k \cdot u_v \| =& \|\sum_{i=1}^n c_i \cdot a \cdot u_i\|\\
  \leq & \sum_{i=1}^n \|a \cdot u_i\|= \sum_{i=1}^n e^{\langle \lambda_i, X\rangle}\| u_i\|\leq e^{C_v \cdot \|X\|},
  \end{aligned}
  \end{equation}
  where 
  \[
  C_v = n \cdot \sup_{Y \in \ka, \text{with} \  \|Y\| = 1 }\big\{\langle \lambda_i, Y \rangle \big| 1 \leq i \leq n \big\}.
  \]
  This completes the proof. 
\end{proof}
\end{proposition}

Now let us fix a parabolic subgroup $P = MAN$. To prove the integral in the definition of $\Phi_{P,x}$ defines a continuous cochain on $\mathcal{S}(G)$, we consider a family of Banach subalgebras $\mathcal{S}_t(G)$, $t\in [0,\infty]$, of $C^*_r(G)$, which was introduced and studied by Lafforgue, \cite[Definition 4.1.1]{MR1914617}. 

\begin{definition}\label{defn:nut}  For $t\in [0, \infty]$, let $\mathcal{S}_t(G)$ be the completion of $C_c(G)$ with respect to the norm $\nu_t$ defined as follows, 
\[
\nu_{t}(f) :=\sup _{g \in G}\Big\{(1+\|g\|))^{t} \Xi(g)^{-1} \big| f(g)\big|\Big\}.
\]
\end{definition}

\begin{proposition}\label{prop:St}
The family of Banach spaces $\{\mathcal{S}_t(G)\}_{t\geq 0}$ satisfies the following properties. 
\begin{enumerate}
\item \label{prop:subalgebra}For every $t\in [0,\infty)$, $\mathcal{S}_t(G)$ is a dense subalgebra of $C^*_r(G)$ stable under holomorphic functional calculus.
\item \label{prop:norms}For $0\leq t_1<t_2<\infty$, $\|f\|_{t_1}\leq \|f\|_{t_2}$, for $f\in \mathcal{S}_{t_2}(G)$. Therefore, $\mathcal{S}(G)\subset \mathcal{S}_{t_2}(G)\subset \mathcal{S}_{t_1}(G)$. 
\item \label{prop:integration-KN}There exists a number $d_0>0$ such that the integral 
\[
f\mapsto f^P(ma):= \int_{KN}F(kmank^{-1})
\]
is a continuous linear map from $\mathcal{S}_{t+d_0}(G)$ to $\mathcal{S}_t(MA)$ for $t\in [0,\infty)$. 
\item \label{prop:orbit-integral} There exists $T_0>0$ such that the orbit integral 
\[
f\mapsto \int_{G/Z_G(x)} f(gxg^{-1})
\]
is a continuous linear functional on $\mathcal{S}_t(G)$ for $t\geq T_0$, $\forall x\in G$. 
\end{enumerate} 
\end{proposition}

\begin{proof} Property \ref{prop:subalgebra} is  from \cite[Proposition 4.1.2]{MR1914617}; Property \ref{prop:norms} follows from the definition of the norm $\nu_t$; Property \ref{prop:integration-KN} follows from \cite[Lemma 21]{MR0219666}; Property \ref{prop:orbit-integral} follows from \cite[Theorem 6]{MR0219666}
\end{proof}

\begin{theorem}\label{thm:definition-cocycle}For any $f_0, \dots f_n \in \mathcal{S}_{T_0+d_0+1}(G)$ for $t\geq T$, and $x \in M$, the following integral
	\[
	\begin{aligned}
	&\int_{h\in M /Z_M(x)} \int_{KN}\int_{G^{\times n}} H_1(g_1 k) \dots H_n(g_n k) \\
	&f_0\left(k hxh^{-1} n k^{-1} (g_1 \dots g_n)^{-1}\right) \cdot f_1(g_1 )\dots f_n(g_n)  
	\end{aligned}
	\]
	is finite, and defines a continuous $n$-linear functional on $\mathcal{S}_{d_0+T_0+1}(G)$. 
	\begin{proof}
	We put 
\[
\tilde{f}_i(g_i) = \sup_{k \in K} \big\{ \big|H_i(g_i k)f_i(g_i)\big| \big\}. 
\]
By Proposition \ref{H est}, we find constants $C_i > 0$ so that 
\[
|H_i(g_i k)|\leq C_i \|g_i k\| = C_i \|g_i\|.
\] 
It shows from  Definition \ref{defn:nut} that $\tilde{f}_i$ belongs to $\mathcal{S}_{d_0+T_0}(G)$, $i=1,...n$. Thus, the integration in (\ref{cocycle equ}) is bounded by the following 
\begin{equation}\label{eq:integral}
	\begin{aligned}
	\int_{h\in M /Z_M(x)} \int_{KN}\int_{G^{\times n}} 
	& \left| f_0\left(k hxh^{-1} n k^{-1} (g_1 \dots g_n)^{-1}\right) \cdot \tilde{f}_1(g_1 )\dots \tilde{f}_n(g_n) \right|\\
	=& \int_{h\in M /Z_M(x)} \int_{KN} F(khxh^{-1}n k^{-1})
	\end{aligned}
	\end{equation}
where by Prop. \ref{prop:St}.\ref{prop:norms}
\[
F = \big| f_0 \ast \tilde{f}_1 \ast\dots \ast \tilde{f}_n \big|\in \mathcal{S}_{d_0+T_0}(G) . 
\]
For any $m \in M, a \in A$, we introduce 
\[
F^{(P)}(m a) = \int_{KN}F(kmank^{-1}). 
\]
By Prop. \ref{prop:St}.\ref{prop:integration-KN}, we have that $F^{(P)}$ belongs to $\mathcal{S}_{T_0}(MA)$. Applying Prop. \ref{prop:St}.\ref{prop:orbit-integral} to the group $MA$, we conclude the orbital integral
\[
\int_{M/ Z_M(x)}F^{(P)}(hxh^{-1}) < +\infty,
\]
from which we obtain the desired finiteness of the integral (\ref{eq:integral}). Furthermore, with the continuity of the above maps, 
\[
f_i\mapsto \tilde{f}_i, \qquad f_0\otimes \tilde{f}_1\otimes ...\otimes \tilde{f}_n\mapsto F,\qquad F\mapsto F^{(P)},\qquad F^{(P)}\mapsto \int_{M/ Z_M(x)}F^{(P)}(hxh^{-1}) ,
\] 
we conclude that the integral (\ref{eq:integral}) is a continuous $n$-linear functional on $\mathcal{S}_{d_0+T_0+1}(G)$.  
	\end{proof}
\end{theorem}

\section{Characters of  representations of $G$}
\label{app:discrete}

\subsection{Discrete series representation of $G$}\label{subsec:discreteseries}

Suppose that $\rank G = \rank K$. Then $G$ has a compact Cartan subgroup $T$ with Lie algebra denoted by $\kt$. We choose a set of positive roots:
\[
 \mathcal{R}^+(\kt, \kg) =  \mathcal{R}^+_c(\kt, \kg)  \cup  \mathcal{R}^+_n(\kt, \kg), 
\]
and define 
\[
\rho_c =\frac{1}{2} \sum_{ \alpha \in \mathcal{R}^+_c(\kt, \kg) }\alpha, \hspace{5mm} \rho_n =\frac{1}{2} \sum_{ \alpha \in \mathcal{R}^+_n(\kt, \kg) }\alpha, \hspace{5mm} \rho = \rho_c + \rho_n. 
\]
The choice of $\mathcal{R}^+_c(\kt, \kk)$ determines a positive Weyl chamber $\kt_+^*$.  Let $\Lambda_T^*$ be the weight lattice in $\kt^*$. Then the set  
\[
\Lambda^*_{K} = \Lambda^*_T \cap \kt^*_+
\]
parametrizes  the set of irreducible $K$-representations. In addition, we denote by $W_K$ the Weyl group of the compact subgroup $K$. For any $w \in W_K$, let $l(w)$ be the length of $w$ and we denote by $(-1)^w = (-1)^{l(w)}$. 

\begin{definition}\label{defn:regular-singular}
Let  $\lambda \in \Lambda^*_{K} +\rho_c$. We say that  $\lambda$ is \emph{regular} if
\[
\langle \lambda, \alpha \rangle \neq 0 
\] 
for all $\alpha \in \mathcal{R}_n(\kt, \kg)$. Otherwise, we say $\lambda$ is \emph{singular}. 
\end{definition}

 Assume that $q = \frac{\dim G/K}{2}$ and $T^\text{reg} \subset T$ the set of regular elements in $T$.
\begin{theorem}[Harish-Chandra]
\label{ds character}
For any regular $\lambda \in \Lambda^*_{K} +\rho_c$, there is a discrete series representation $\sigma(\lambda)$ of $G$ with Harish-Chandra parameter $\lambda$. Its character is given by the following formula:
\[
\Theta( \lambda) \big|_{T^{\reg}} = (-1)^q  \cdot \frac{\sum_{w\in W_K} (-1)^w e^{w \lambda}}{\Delta^G_T},
\]
where
\[
\Delta^G_T = \prod_{\alpha \in \mathcal{R}^+(\kt, \kg)} (e^{\frac{\alpha}{2}}-e^{\frac{-\alpha}{2}} ).
\]
\end{theorem}

Next we consider the case when $\lambda \in \Lambda^*_{K} + \rho_c$ is singular. That is, there exists at least one noncompact root $\alpha$ so that $\langle \lambda, \alpha \rangle = 0$. 
Choose a positive root system $\mathcal{R}^+(\kt, \kg)$ that makes $\lambda$ dominant; the choices of $\mathcal{R}^+(\kt, \kg)$ are not unique when $\lambda$ is singular. For every choice of $\mathcal{R}^+(\kt, \kg)$, we can associate it with a representation, denoted by $\sigma \big(\lambda, \mathcal{R}^+\big)$. We call $\sigma(\lambda, \mathcal{R}^+)$ a \emph{limit of discrete series representation} of $G$.  Distinct choices of $\mathcal{R}^+(\kt, \kg)$ lead to infinitesimally equivalent versions of $\sigma \big(\lambda, \mathcal{R}^+\big)$. Let $\Theta\big(\lambda, \mathcal{R}^+\big)$ be the character of $\sigma\big(\lambda, \mathcal{R}^+\big)$. Then 
\[
\Theta\big(\lambda, \mathcal{R}^+\big)\big|_{T^{\reg}} =  (-1)^\pm \frac{\sum_{w\in W_K} (-1)^w e^{w \lambda}}{\Delta^G_T}.
\]
Moreover, for any $w \in W_K$ which fixes $\lambda$, we have that 
\begin{equation}
\label{char of limit}
\Theta\big(\lambda, w\cdot \mathcal{R}^+\big)\big|_{T^{\reg}} =  (-1)^w \cdot \Theta\big(\lambda, w\cdot \mathcal{R}^+\big)\big|_{T^{\reg}}. 	
\end{equation}
See \cite[P. 460]{MR1880691} for more detailed discussion.

\subsection{Discrete series representations of $M$}
Let $P = MAN$ be a parabolic subgroup. The subgroup $M$ might not be connected in general.  We denote by $M_0$ the connected component of $M$ and set 
\[
M^\sharp = M_0 Z_M,
\]  
where $Z_M$ is the center for $M$. 

Let $\sigma_0$  be a discrete series representation (or limit of discrete series representation) for the connected group $M_0$ and $\chi$ be a unitary character of $Z_M$. If $\sigma_0$ has Harish-Chandra parameter $\lambda$, then we assume that   
\[
\chi\big|_{T_M \cap Z_M} = e^{\lambda - \rho_M}\big|_{T_M \cap Z_M}. 
\]
We have the well-defined representation $\sigma_0 \boxtimes \chi$ of $M^\sharp$, given by 
\[
\sigma_0 \boxtimes \chi(gz) = \sigma(g) \chi(z),
\]
for $g \in M_0$ and $z \in Z_M$. 

\begin{definition}\label{defn:discrete-M}
The discrete series representation or limit of discrete series representation  $\sigma$ for the possibly disconnected group $M$ induced from $\sigma_0\boxtimes \chi$ is defined as
\[
\sigma= \Ind^M_{M^\sharp}\big( \sigma_0  \boxtimes \chi\big).
\]
\end{definition}
Discrete series representations of $M$ are parametrized by a pair of Harish-Chandra parameter $\lambda$ and unitary character $\chi$. Next we show that $\chi$ is redundant for the case of $M_\cc$.  Denote
\begin{itemize}
\item 
$\ka =$ the Lie algebra of $A$;
\item 
$\kt_M = $ the Lie algebra of the compact Cartan subgroup of $M$;
\item
$\ka_M = $ the maximal abelian subalgebra of $\kp \cap \km$, where $\kg = \kk \oplus \kp$;
\end{itemize}
Then $\kt_M \oplus \ka$  is a Cartan subalgebra of $\kg$, and $\ka_\kp = \ka_M \oplus \ka$ is a maximal abelian subalgebra in $\kp$.

Let $\alpha$ be a real root in $\mathcal{R}(\kg, \kt_M \oplus \ka)$. Restrict $\alpha$ to $\ka$ and extend it by $0$ on $\ka_M$ to obtain a restricted root in $\mathcal{R}(\kg, \ka_\kp)$. Form an element $H_\alpha \in \ka_\kp$ by the following
\[
\alpha(H) = \langle H, H_\alpha \rangle, \hspace{5mm} H \in \ka_\kp.   
\]
It is direct to check that 
\[
\gamma_\alpha = \exp \Big( \frac{2\pi i H_\alpha}{|\alpha|^2} \Big)
\]
is a member of the center of $M$. Denote by $F_M$ the finite group generated by all $\gamma_\alpha$ induced from real roots of $\Delta(\kg, \kt_M \oplus \ka)$. It follows from Lemma 12.30 in \cite{MR1880691} that 
\begin{equation}
\label{M sharp}
M^\sharp = M_0 F_M.  
\end{equation}

\begin{lemma}
For the maximal parabolic subgroup $P_\cc = M_\cc A_\cc N_\cc$, we have that 
\[
Z_{M_\cc} \subseteq (M_\cc)_0. 
\]
\begin{proof}
There is no real root in $\mathcal{R}(\kh_\cc, \kg)$ since the Cartan subgroup $H_\cc$ is maximally compact. The lemma follows from (\ref{M sharp}). 
\end{proof}
\end{lemma}

It follows that discrete series or limit of discrete series representation of $M_\cc$ are para-metrized by Harish-Chandra parameter $\lambda$. We denote them by $\sigma(\lambda)$ or $\sigma(\lambda, \mathcal{R}^+)$.

\subsection{Induced representations of $G$}
Let $P = MAN$ be a parabolic subgroup of $G$ and $L = MA$ its Levi subgroup as before. For any Cartan subgroup $J$ of $L$, let $\{ J_1, J_2, \dots, J_k \}$ be a complete set of representatives for distinct conjugacy classes of Cartan subgroups of $L$ for which $J_i$ is conjugate to $J$ in $G$. Let $x_i \in G$ satisfy $J_i = x_i J x_i^{-1}$ and for $j \in J$, write $j_i = x_i j x_i^{-1}$. 
\begin{theorem}\label{thm:B.3}
Let $\Theta(P, \sigma, \varphi)$ be the character of the basic representation $\Ind_P^G(\sigma \otimes \varphi)$. Then
\begin{itemize}
\item
$\Theta(P, \sigma, \varphi)$ is a locally integrable function.
\item
$\Theta(P, \sigma, \varphi)$ is nonvanishing only on Cartan subgroups of $G$ that are $G$-conjugate to Cartan subgroups of $L$. 
\item
For any $j \in J$, we have 
\begin{equation}
\label{induced character eq}
\begin{aligned}
\Theta(P, \sigma, \varphi)(j) = &\sum_{i=1}^k  |W( J_i, L)|^{-1} |\Delta^G_{J_i}(j_i)|^{-1} \times \\
& \Big( \sum_{w \in W(J_i, G)}|\Delta^L_{J_i}(wj_i)| \cdot  \Theta^{M}_\sigma\big(wj_i \big|_{M}\big) \varphi(wj_i|_{H_p}) \Big), 
\end{aligned}
\end{equation}
where $\Theta^{M}_\sigma$ is the character for the $M_P$ representation $\sigma$, and the definition of $\Delta^G_{J_i}$ (and $\Delta^L_{J_i}$) is explained in Theorem \ref{ds character}. 
\end{itemize}
\begin{proof}
The first two properties of $\Theta(P, \sigma, \varphi)$ can be found in \cite{MR1880691}[Proposition 10.19], and the last formula has been given in 
\cite{MR648480}[Equation (2.9)].
\end{proof}
\end{theorem}

\begin{corollary}
\label{coro character}
Suppose that $P_\cc$ is the maximal parabolic subgroup of $G$ and $\sigma^{M_\cc}(\lambda)$ is a (limit of) discrete series representation with Harish-Chandra parameter $\lambda$. We have that  
\[
\Theta \left(P_\cc, \sigma^{M_\cc}\left(\lambda\right), \varphi \right)(h) =\frac{\sum_{w \in K} (-1)^w e^{w\lambda}(h_k) \cdot  \varphi(h_p)}{\Delta^G_{H_\cc}(h)}. 
\]	
for any $h \in H_\cc^\text{reg}$.
\begin{proof}
The corollary follows from (\ref{induced character eq}) and Theorem \ref{ds character}. 
\end{proof}

\end{corollary}

\section{Description of $K(C^*_r(G))$}
\label{sec: description}
Without loss of generality, we assume that $\dim A_\cc = n $ is even. Otherwise, we can replace $G$ by $G \times \R$. 

\subsection{Generalized Schmid identity}
Suppose that $P = MAN$ is a  parabolic subgroup of $G$ and $H = TA$ is its associated Cartan subgroup. We assume that $P$ is not maximal and thus $H$ is not the most compact. By Cayley transform, we can obtain a more compact Cartan subgroup $H^*= T^* A^*$. We denote by $P^* = M^* A^* N^*$ the corresponding parabolic subgroup. Here $A =  A^*  \times \R$.

Let $\sigma$ be a (limit of) discrete series representation of $M$, and
\[
 \nu \otimes 1  \in \widehat{A} = \widehat{A^*}  \times \widehat{\R}.
\] 
Suppose that 
\[
\pi = \Ind_P^G\big(\sigma \otimes (\nu \otimes 1)\big)
\] 
is a basic representation. Then $\pi$ is either  irreducible or decomposes as follows,
\[
\Ind_P^G\big(\sigma \otimes (\nu \otimes 1)\big) = \Ind^G_{P^*}( \delta_1 \otimes \nu) \oplus \Ind^G_{P^*}(\delta_2 \otimes \nu).
\]
Here $\delta_1$ and $\delta_2$ are limit of discrete series representations of $M^*$. Moreover, they share the same Harish-Chandra parameter but corresponds to different choices of positive roots. On the right hand side of the above equation, if $P^*$ is not maximal, then one can continue the decomposition for $\Ind^G_{P^*}(\sigma_i^* \otimes \nu), i = 1, 2$. Eventually, we get
\begin{equation}
\label{schmid identity}
\Ind_P^G\big(\sigma \otimes (\varphi \otimes 1)\big) = \bigoplus_{i} \Ind^G_{P_\cc}( \delta_i \otimes \varphi),
\end{equation}
where 
\[
\varphi \otimes 1 \in \widehat{A}_P =\widehat{A}_\cc \times  \widehat{A}_S.
\] 
The number of component in the above decomposition is closely related to the $R$-group which we will discuss below.  We refer to  \cite[Corollary 14.72]{MR1880691} for detailed discussion.   

As a consequence, we obtain the following lemma immediately. 
\begin{lemma}
Let $P_\cc = M_\cc A_\cc N_\cc$ be the maximal parabolic subgroup. If $\sigma \otimes \varphi$ is an irreducible representation of $ M_\cc A_\cc$, then the induced representation
\[
\Ind_{P_\cc}^G \big( \sigma \otimes \varphi\big)
\]
is also irreducible. 
\end{lemma}

\subsection{Essential representations}
Clare-Crisp-Higson proved in \cite{MR3518312}[Section 6] that the group $C^*$-algebra $C^*_r(G)$ has the following decomposition:
\[
C^*_r(G) \cong \bigsqcup_{[P, \sigma] \in \mathcal{P}(G)} C^*_r(G)_{[P, \sigma]},
\]
where 
\[
C^*_r(G)_{[P, \sigma]} \cong \mathcal{K}\big(\Ind_P^G(\sigma)\big)^{W_\sigma}. 
\]
For principal series representations $\Ind_P^G(\sigma \otimes \varphi)$, Knapp and Stein \cite[Chapter 9]{MR1880691} showed that the stabilizer $W_\sigma$ admits a semidirect product decomposition
\[
W_\sigma = W'_\sigma \rtimes R_\sigma, 
\]
where the \emph{R-group} $R_\sigma$ consists of those elements that actually contribute nontrivially to the intertwining algebra of $\Ind_P^G(\sigma \otimes \varphi)$. Wassermann notes the following Morita equivalence,
\[
\mathcal{K}\big(\Ind_P^G(\sigma)\big)^{W_\sigma} \sim C_0(\widehat{A}_P/W'_\sigma) \rtimes R_\sigma. 
\]

\begin{definition}
We say that an equivalence class $[P, \sigma]$ is \emph{essential} if $W_\sigma = R_\sigma$. We denote it by $[P, \sigma]_{\ess}$. In this case, 
\[
W_\sigma = R_\sigma \cong (\Z_2)^{r}
\]
is obtained by application of all combinations of $r = \dim(A_P)- \dim (A_\cc)$ commuting reflections in simple noncompact roots. 
\end{definition}

As before, let $T$ be the maximal torus of $K$. We denote by  $\Lambda_T^*$ and   $\Lambda_{K}^*$ the weight lattice and its intersection with the positive Weyl chamber of $K$. 
\begin{theorem}(Clare-Higson-Song-Tang-Vogan)
\label{Schmid-iden}
 There is a bijection between the set of $[P, \sigma]_{\ess}$ and  the set $\Lambda^*_{K}+\rho_c$. Moreover, 
 Suppose that $[P, \sigma]$ is essential. 
\begin{itemize}
\item
If $\lambda$ is regular, that is, 
\[
\langle \lambda, \alpha \rangle \neq 0
\]
for all non-compact roots $\alpha \in \mathcal{R}_n$, then  $W_\sigma$ is trivial, $P = P_\cc$, and $\sigma$ is the discrete series representation of $M_\cc$ with Harish-Chandra parameter $\lambda$. In addition, 
\[
 \Ind_{P_\cc}^G \left(\sigma \otimes \varphi \right)  
\]
are irreducible for all $\varphi \in \widehat{A}_\cc$. 
\item
Otherwise, if $\langle \lambda, \alpha \rangle= 0$ for some $\alpha \in \mathcal{R}_n$, then
\begin{equation}
\label{schmid dec}
 \Ind_{P}^G(\sigma \otimes \varphi \otimes 1) = \bigoplus_{i=1}^{2^r} \Ind_{P_\cc}^G \big(\delta_i \otimes \varphi \big),\end{equation}
where  $\delta_i$ is a limit of discrete series representation of $M_\cc$ with Harish-Chandra parameter $\lambda$, $\varphi \in \widehat{A}_\cc$ and $\varphi \otimes 1 \in \widehat{A}_P$.  
\end{itemize}
 \end{theorem}

The computation of $K$-theory group of $C_r^*(G)$ can be summarized as follows. 

\begin{theorem}\label{thm:k-theory-G} (Clare-Higson-Song-Tang) The $K$-theory group of $C_r^*(G)$ is a free abelian group generated by the following components, i.e.
\begin{equation}\label{eq:k-theory}
\begin{aligned}
K_0(C^*_r(G))\cong& \bigoplus_{[P, \sigma]^{\ess}} K_0 \left(\mathcal{K}\big(C^*_r(G)_{[P, \sigma]}\right)\\
\cong& \bigoplus_{[P, \sigma]^{\ess}} K_0 \left(\mathcal{K}\big(\Ind_P^G \sigma \big)^{W_\sigma}\right)\\
\cong& \bigoplus_{ \text{regular  \ part}}K_0 \big(C_0(\R^{n})\big) \oplus \bigoplus_{ \text{singular \ part}}K_0 \Big( \big(C_0(\R)\rtimes \Z_2\big)^r \otimes C_0(\R^{n})\Big) \\
\cong & \bigoplus_{ \lambda \in \Lambda^*_{K} +\rho_c} \mathbb{Z}.
\end{aligned}
\end{equation}
\end{theorem}

\begin{example}
Let $G = SL(2, \R)$. The principal series representations of $SL(2, \R)$ are para-metrized by characters
\[
(\sigma, \lambda) \in \widehat{M A} \cong\{ \pm 1\} \times \mathbb{R}
\]
modulo the action of the Weyl group $\Z_2$. One family of principal series representations is irreducible at $0$ while the other decomposes as a sum of two limit of discrete series  representations. At the level of $C^*_r(G)$, this can be explained as
\[
\begin{aligned}
 \widehat{M A} / \mathbb{Z}_{2} \cong\{+1\} \times[0, \infty) & \cup\{-1\} \times[0, \infty) \\ & \cong\{+1\} \times \mathbb{R} / \mathbb{Z}_{2} \cup\{-1\} \times \mathbb{R} / \mathbb{Z}_{2}, 
 \end{aligned}
 \]
 and the principal series contribute summands to $C^*_r (SL(2,\R))$ of the form
 \[
 C_{0}\left(\mathbb{R} / \mathbb{Z}_{2}\right) \quad \text { and } \quad C_{0}(\mathbb{R}) \rtimes \mathbb{Z}_{2}
 \]
 up to Morita equivalence.  In addition $SL(2, \R)$ has discrete series representations each of which contributes
a summand of $\C$ to $C^*_r (SL(2,\R))$, up to Morita equivalence. We obtain:
\[
C_{r}^{*} (S L(2, \mathbb{R})) \sim C_{0}\left(\mathbb{R} / \mathbb{Z}_{2}\right) \oplus C_{0}(\mathbb{R}) \rtimes \mathbb{Z}_{2} \oplus \bigoplus_{n \in \mathbb{Z} \backslash\{0\}} \mathbb{C}.
\]
Here the part $ C_{0}\left(\mathbb{R} / \mathbb{Z}_{2}\right)$ corresponds to the family of spherical principal series representations, which are not essential. Then (\ref{eq:k-theory}) can be read as follows,
\[
\begin{aligned}
K_0(C^*_r(SL(2,\R)))\cong K_0 \Big( \big(C_0(\R)\rtimes \Z_2\big) \Big)  \oplus \bigoplus_{n \neq 0}K_0 \big(\C\big). 
\end{aligned}
\]
\end{example}

\section{Orbital integrals}
\label{sec: orbital integral}
In this section, we assume that $\rank G = \rank K$. We denote by $T$  the compact Cartan subgroup of $G$. 
\subsection{Definition of orbital integrals}
Let $h$ be any semisimple element of $G$, and let $Z_G(h)$ denote its centralizer in $G$. Associated with $h$ is an
invariant distribution $\Lambda_f(h)$ given for $f \in \mathcal{S}(G)$ by 
\[
f \mapsto \Lambda_f(h) = \int_{G/Z_G(h)} f(gh g^{-1}) d_{G/Z_G(h)}\dot{g},
\]
where $d\dot{g}$ denotes a left $G$-invariant measure on the quotient $G/Z_G(h)$. In this paper, we are only interested in two cases: $h = e$ and 
$h$ is regular.

If $h = e$ is the identity element, the formula for $\Lambda_f(e)$ is the Harish-Chandra's Plancherel formula for $G$.
\begin{theorem}[Harish-Chandra's Plancherel formula]
\label{Planchere formula}
For any $f \in \mathcal{S}(G)$,
\[
\Lambda_f(e) = f(e) =\int_{\pi \in \widehat{G}_\mathrm{temp}} \Theta(\pi)(f) \cdot  m(\pi)d\pi
\]
where $m(\pi)$ is the Plancherel density and $\Theta(\pi)$ is the character for the irreducible tempered representation $\pi$.
\end{theorem}

Suppose that $h$ is regular, then $Z_G(h) = H$ is a Cartan subgroup of $G$. We define 
\[
\Delta^G_H = \prod_{\alpha \in \mathcal{R}^+(\kh, \kg)}(e^{\frac{\alpha}{2}} - e^{-\frac{\alpha}{2}}).
\]
We can normalize the measures on $G/H$ and $H$ so that 
\[
\int_G f(g) dg = \int_{G/H}\int_H f(g h) dh d_{G/H}\dot{g}. 
\]

\begin{definition}
For any $t \in T^\text{reg}$, the \emph{orbital integral} is defined by 
\[
F^T_f(t) = \Delta^G_T(t) \int_{G/T} f(g tg^{-1}) d_{G/T} \dot{g}. 
\]
Similarly, if $h \in H^\text{reg}$, we define
\[
F^H_f(h) =\epsilon^H(h)  \cdot  \Delta^G_H(h) \int_{G/H} f(g hg^{-1}) d_{G/H} \dot{g}, 
\]
where $\epsilon^H(h)$ is a sign function defined in \cite[P. 349]{MR1880691}.
\end{definition}	
Note that $F^H_f$ is anti-invariant under the Weyl group action, that is, for any element $w \in W(H, G)$, 
\[
F^H_f(w\cdot h) = (-1)^w \cdot F^H_f(h). 
\]

\subsection{The formula for orbital integrals}
\label{section orbital}
In this subsection, we summarize the formulas and results in \cite{MR525674, MR648480}. If $P$ is the minimal parabolic subgroup with the most non-compact Cartan subgroup $H$, then the Fourier transform of orbital integral equals to the character of representation. That is, for any $h \in H^\mathrm{reg}$
\[
\widehat{F}^H_f(h) = \int_{\chi \in \widehat{H}} \chi(h) \cdot F^{H}_f(h) \cdot d\chi= \Theta(P, \chi)(f)
\]
or equivalently, 
\[
F^H_f(h) = \int_{\chi \in \widehat{H}}  \Theta(P, \chi)(f)\cdot \overline{\chi(h)} \cdot  d\chi. 
\]
For any arbitrary parabolic subgroup $P$, the formula for orbital integral is much more complicated, given as follows, 

\begin{equation}
\label{kappa thm}
F_f^{H}(h) = \sum_{Q \in \mathrm{Par}(G, P)} \int_{\chi \in \widehat{J}} \Theta(Q, \chi)(f) \cdot \kappa^G(Q, \chi, h) d\chi,
\end{equation}
where 
\begin{itemize}
\item
the sum ranges over the set
\[
\mathrm{Par}(G, P) = \big\{ \text{parabolic subgroup $Q$ of  $G$}\big|  Q \ \text{is no more compact than} \ P \big\}.
\] 
\item
$J$ is the Cartan subgroup associated to the parabolic subgroup $Q$;
\item
$\chi$ is a unitary character of $J$ and $\Theta(Q, \chi)$  is a tempered invariant eigen-distribution defined in \cite{MR525674}. In particular, $\Theta(Q, \chi)$ is the character of parabolic induced representation or an alternating sum of characters which can be embedded in a reducible unitary principal series representation associated to a different  parabolic subgroup;
\item
The function $\kappa^G$ is rather complicated to compute. Nevertheless, for the purpose of this paper, we only need to know the existence of functions $\kappa^G$, which has been verified in \cite{MR0450461}. 
\end{itemize}

In a special case when $P = G$ and $H =T$, the formula (\ref{kappa thm}) has the following more explicit form.
\begin{theorem}
\label{orbital formula}
For any $t \in T^{\text{reg}}$, the orbital integral 
\begin{equation}
\label{formula orbit}	
\begin{aligned}
F^{T}_f(t) =& \sum_{ \mathrm{regular} \  \lambda \in \Lambda^*_{K} + \rho_c} \sum_{w \in W_K} (-1)^w \cdot e^{w \cdot \lambda}(t)\cdot   \Theta(\lambda)(f) \\
+& \sum_{ \mathrm{singular} \  \lambda \in \Lambda^*_{K} + \rho_c} \sum_{w \in W_K} (-1)^w \cdot e^{w \cdot \lambda}(t)\cdot   \Theta(\lambda)(f)\\
+&  \int_{\pi \in \widehat{G}^\mathrm{high}_\mathrm{temp}} \Theta(\pi)(f) \cdot \kappa^G(\pi, t) d\chi. 
\end{aligned}
\end{equation}
In the above formula, there are three parts:
\begin{itemize}
\item
regular part: $\Theta(\lambda)$ is the character of the discrete series representation with Harish-Chandra parameter $\lambda$;
\item
singular part: for singular $\lambda \in \Lambda^*_{K} + \rho_c$, we denote by $n(\lambda)$ the number of different limit of discrete series representations with Harish-Chandra parameter $\lambda$. By (\ref{char of limit}), we can organize them so that 
\[
\Theta_1(\lambda)\big|_{T^{\text{reg}}} = \dots = \Theta_{\frac{n(\lambda)}{2}}(\lambda)\big|_{T^{\text{reg}}} = -\Theta_{\frac{n(\lambda)}{2}+1}(\lambda)  \big|_{T^{\text{reg}}}= \dots = -\Theta_{n(\lambda)}(\lambda)\big|_{T^{\text{reg}}}. 
\]
We put 
\[
\Theta(\lambda) \colon = \frac{1}{n(\lambda)} \cdot \Big( \sum_{i=1}^{\frac{n(\lambda)}{2}} \Theta_i(\lambda) - \sum_{i=\frac{n(\lambda)}{2} + 1}^{n(\lambda)} \Theta_i(\lambda)\Big). 
\]
\item
higher part:  $\widehat{G}^\mathrm{high}_\mathrm{temp}$ is a subset of $\widehat{G}_\mathrm{temp}$ consisting of irreducible tempered representations which are not (limit of) discrete series representations.  
\end{itemize}

\end{theorem}

\bibliographystyle{plain}
\bibliography{mybib}

\end{document}